\documentclass[12pt,a4paper]{article}

\usepackage[utf8]{inputenc}
\usepackage[english]{babel}
\usepackage{amsmath}
\usepackage{amsfonts}
\usepackage{amssymb}
\usepackage{amsthm}
\usepackage{algorithm2e}
\usepackage{bbm}
\usepackage{mathrsfs}
\usepackage{graphicx}
\usepackage[dvipsnames]{xcolor}
\usepackage{lmodern}

\usepackage[final]{changes}
\setdeletedmarkup{\color{red}{\sout{#1}}}

 \usepackage{todonotes}
\DeclareMathAlphabet{\mathpzc}{OT1}{pzc}{m}{it}
\usepackage{fourier}
\usepackage {subfig}
\usepackage[section]{placeins}
\usepackage{booktabs}
\usepackage{wrapfig}
\usepackage{url}
\usepackage{pdfpages}
\usepackage{float}
\usepackage{qtree}
\usepackage[shortlabels]{enumitem}
\usepackage{listings}
\usepackage[linktocpage,hidelinks]{hyperref}

\usepackage[autooneside=false]{scrlayer-scrpage}
\usepackage{tikz}
\usepackage{pgfplots}
\usetikzlibrary{decorations.text}
\usetikzlibrary{shapes}
\usetikzlibrary{trees}
\usetikzlibrary{calc}
\usepgfplotslibrary{fillbetween}
\ihead{\headmark}
\ohead{\pagemark}
\newtheorem{defi}{Definition}
\newtheorem{theorem}[defi]{Theorem}

\newcommand{\Oad}{\mathcal{O}^{\text{ad}}}

\newcommand{\R}{\mathbb{R}}	
\newcommand{\eps}{\varepsilon}	
\newcommand{\Kcrit}{K_{I_c}}

\usepackage{nomencl}
\makenomenclature
\newcommand{\spur}{\text{tr }}
\nomenclature{\spur}{trace}

\nomenclature{$f_1$}{reliability objective function}
\nomenclature{$f_2$}{volume objetive function }
\nomenclature{$f = (f_1,f_2)$}{ bicriteria objective function}
\nomenclature{$\Omega\subset\R^2$}{compact body filled with ceramic material }
\nomenclature{$\mathcal{O}$}{Set of admissible shapes ($\Omega\in\mathcal{O}$)}
\nomenclature{$\R$}{real numbers}
\nomenclature{$\partial \Omega$}{boundary of $\Omega$}
\nomenclature{$\partial \Omega_D$}{boundary of $\Omega$ with Dirichlet cond. }
\nomenclature{$\partial \Omega_{N_{fixed}}$}{boundary of $\Omega$ with Neumann cond., surface forces may act}
\nomenclature{$\partial \Omega_{N_{free}}$}{boundary of $\Omega$with Neumann cond., surface force-free}
\nomenclature{$\overline{\Omega}$}{Abschluss}
\nomenclature{$\tilde{f}:\Omega\rightarrow \R^2$}{function describing volume forces}
\nomenclature{$\tilde{g}:\partial\Omega\rightarrow \R^2$}{function describing surface forces}
\nomenclature{$u:\Omega\rightarrow\R^2$}{twice differentiable, describing replacement}
\nomenclature{$\sigma$}{elastic stress field}
\nomenclature{$K_{I}$}{stress intensity factor corresponding to compressive and tensile load}
\nomenclature{$K_{II}$}{stress intensity factor }
\nomenclature{$K_{III}$}{stress intensity factor }
\nomenclature{$\sigma_n$}{tensile loading in a normal direction}
\nomenclature{$a\in\R_+$}{radius of a crack}
\nomenclature{$\R_+$}{positive numbers}
\nomenclature{$\Kcrit$}{critical K-factor}
\nomenclature{$\eps (u) = \frac{1}{2}(Du+Du^{\top})$}{elastic strain field}
\nomenclature{$^{\top}$}{transpose}
\nomenclature{$\lambda,\mu$}{Lam\'e constants}
\nomenclature{$S^1$}{unit sphere in $\R^2$}
\nomenclature{$n\in S^1$}{orientation of a crack}
\nomenclature{$\mathscr{C}=(\Omega\times S^1 \times (0,\infty))$}{crack configuration space}
\nomenclature{$A_c$}{set of critical crack configurations}
\nomenclature{$Df$}{Jacobi matrix of $f$}
\nomenclature{$\frac{df}{d ?}$}{partial derivative ?}
\nomenclature{$\omega \in (0,1)$}{weighting vector}
\nomenclature{$\nu^P$}{Poisson's ratio}
\nomenclature{$\texttt{E}$}{Young's modulus}
\nomenclature{$\upsilon(A_c(\Omega,Du))$}{intensity measure}
\nomenclature{$m=5$}{}
\nomenclature{$n_B$}{number of bsplines}

\nomenclature{$P$}{probability measure given by the PPP conditions}
\nomenclature{$N$}{counting variable for the amount of cracks in $A_c$}
\nomenclature{$\nu$}{Radon measure on the sigma algebra $\mathscr{A}(\mathscr{C})$}
\nomenclature{$\rho([c,d])$}{fixes the density of cracks with radius $a$, $c\leq a\leq d$}

\numberwithin{equation}{section} 

\title{Gradient Based Biobjective Shape Optimization to Improve Reliability and Cost of Ceramic Components}
\author{\textsc{O.\ T.\ Doganay,  \added{H.\ Gottschalk,  C.\ Hahn}}\\\textsc{K.\ Klamroth, J.\ Schultes, M.\ Stiglmayr}\\
School of Mathematics and Natural Sciences\\ University of Wuppertal\thanks{\texttt{$\{$doganay,schultes,hahn,stiglmayr,gottschalk,klamroth$\}$@math.uni-wuppertal.de}}}

\begin{document}


\maketitle

\begin{abstract}
We consider the simultaneous optimization of the reliability and the cost of a ceramic component in a biobjective PDE constrained shape optimization problem. A probabilistic Weibull-type model is used to assess the probability of failure of the component under tensile load, while the cost is 
assumed to be proportional to the  volume of the component. Two different gradient-based optimization methods are suggested and compared at 2D test cases. The numerical implementation is based on a first discretize then optimize strategy and benefits from  efficient gradient computations using adjoint equations. The resulting approximations of the Pareto front nicely exhibit  the trade-off between reliability and cost and give rise to innovative shapes that compromise between these conflicting  objectives.
\end{abstract}



\section{Introduction}
\label{sec:Intro}
The optimization of the design of mechanical structures is a central task in mechanical engineering. If the material for a component is chosen and the use cases are defined, implying in particular the mechanical loads, then the central task of engineering design is to define the shape of the component. Among all possible choices, those shapes are preferred that guarantee the desired functionality at minimal cost. The functionality, however, is only guaranteed if the mechanical integrity of the component is preserved. The fundamental design requirements of functional integrity and cost are almost always in conflict, which makes mechanical engineering an optimization problem with at least two objective functions to consider.   

Mathematically, the task of choosing the shape of a structure is formulated by the theory of shape optimization, see e.g., \cite{allaire,bucur,haslinger,sokolowski} for an introduction. We thus consider admissible shapes $\Omega\subseteq \R^p$, $p=2,3$, 
along with an objective function $f(\Omega)$ which returns lower values for better designs. The task then is to  find an admissible shape $\Omega^*\in\arg\min f(\Omega)$. The existence of optimal shapes has been studied in \cite{chenais,fuji,haslinger} -- for the specific objective function $f$ of compliance see \cite{allaire}. On the algorithmic side, the adjoint approach to shape calculus has led to efficient strategies to calculate shape gradients, see e.g., \cite{conti,delfour,eppler2,eppler,laurain,schulz,schulz2,sokolowski}.  While theory and numerical algorithms of shape calculus are highly developed mathematically, most publications in the field neither deal with multiobjective optimization problems, nor directly consider mechanical integrity as one of the objective functions, see \cite{haslinger,allaire:2008,duysinx:1998,picelli:2018}  for some remarkable exceptions.    

In mechanical engineering, mechanical integrity is one of the central objectives, see e.g., \cite{baeker}. However, if objectives like the ultimate load that the structure can bear or the fatigue life of a component are formulated deterministically, then the objective function is in general non differentiable as it only depends on the point of maximal stress. In numerical shape optimization this would lead to shape gradients concentrated on a single node, resulting in highly instable optimization schemes. At the same time, as material properties are subject to considerable scatter, a deterministic approach is not realistic. To overcome these two shortcomings, an alternative probabilistic approach to mechanical integrity has been proposed by some of the authors and others 
\cite{PaperHahn,bolten,gottschalk3,gottschalk2,gottschalk,schmitz,schmitz3,schmitz2}, which has a smoothing effect on the singularities that are typical for deterministic models. Note that the probabilistic description of the ultimate strength of ceramics has become a standard in material  engineering since the ground breaking work of Weibull, see e.g. \cite{baeker,bruecknerfoit,munz,riesch,weibull}.       

In practice, there usually is a trade-off between the mechanical integrity of a  structure and its volume (cost), since an improved mechanical integrity usually comes at the cost of a larger volume. Instead of presetting a fixed bound on the allowable volume, the trade-off between these two conflicting goals 
can be analyzed in a biobjective model. Other relevant objective functions may be, for example, the minimal buckling load of a structure or its minimal natural frequency, see, for example, \cite{haslinger}. For a general introduction into the field of multiobjective optimization we refer to \cite{ehrgott,Miettinen1999}. In the context of shape optimization problems, two major solution approaches can be distinguished: Metaheuristic and, in particular, evolutionary algorithms are widely applicable solution paradigms that do not utilize the particular structure of a given problem \cite{chir:mult:2018,Deb2001,Deb2002,zava:asur:2014}. However, in combination with expensive numerical simulations such approaches tend to be inefficient. On the other hand, gradient-based algorithms \cite{desideri:MGDA,Fliege2000,zerb:meta:2014} require efficient gradient computations and are often applied in the context of adjoint approaches and using weighted sum scalarizations of the objective functions. See \cite{pull:comp:2003} 
for a comparison.

In this paper, we suggest a biobjective PDE constrained shape optimization problem for the simultaneous optimization of the mechanical integrity and the cost of a ceramic component. Section~\ref{sec:Biobjective} is devoted to a formal introduction of the problem, including a review of Weibull type models for the probability of failure and existence results for Pareto optimal shapes. 
The numerical implementation is based on a first discretize then optimize approach using Lagrangian finite elements, which is detailed in Section~\ref{sec:Implementation}. Section~\ref{sec:Implementation} also contains an introduction to gradient-based optimization strategies for biobjective problems and some details on their efficient implementation. The approach is validated  at 2D ceramic components in Section~\ref{sec:CaseStudies}, and  perspectives for future research are suggested in Section~\ref{Sec:Outlook}.

\section{Biobjective Shape Optimization (of Ceramic Structures)}
\label{sec:Biobjective}

is not available in this case. This motivates the formulation of a biobjective shape optimization problem where mechanical integrity and volume (cost) are considered simultaneously as equitable objectives. 
%

In this section, we first introduce a set of admissible (feasible) shapes and review the state equations that model the physical behavior of a shape under external forces according to the linear elasticity theory (Section~\ref{subsec:state_equation}). The considered objective functions, the intensity measure modelling the mechanical integrity, and the volume of the shape, are formally introduced in Sections~\ref{subsec:PoF} and \ref{subsec:volume}, respectively. The overall problem is formulated as a biobjective optimization problem in Section~\ref{subsec:bioptalg}, and 
the existence of Pareto optimal solutions is shown in Section~\ref{subsec:exist}.

\subsection{Admissible Shapes and State Equation}\label{subsec:state_equation}
We follow the description from \cite{PaperHahn,bolten} and consider a compact body (also referred to as component or shape) $\Omega\subset \R^p$, $p=2,3$ 
with Lipschitz boundary that is filled with ceramic material. 
Furthermore, we assume that the boundary $\partial\Omega$ of $\Omega$ is subdivided into three parts with nonempty relative interior, 
\begin{align*}
	\partial\Omega = \text{cl}({\partial\Omega}_D) \cup \text{cl}({\partial\Omega}_{N_{\text{fixed}}}) \cup \text{cl}({\partial\Omega}_{N_{\text{free}}}).
\end{align*}
$\partial\Omega_D$ describes the part of the boundary where the Dirichlet boundary condition holds, $\partial\Omega_{N_{\text{fixed}}}$ the part where surface forces may act on and $\partial\Omega_{N_{\text{free}}}$ the part of the boundary that can be modified in an optimization approach. It is assumed to be force free for technical reasons \cite{bolten}. 

\begin{figure}[htb]
  \centering
  \begin{tikzpicture}[font={\footnotesize}]
    \coordinate (bb_ll) at (0,0.4);
    \coordinate (bb_lr) at (5,0.4);
    \coordinate (bb_tl) at (0,3.4);
    \coordinate (bb_tr) at (5,3.4);
    \coordinate (bar_ll) at (0,1.);
    \coordinate (bar_lr) at (5,1.);
    \coordinate (bar_tl) at (0,1.8);
    \coordinate (bar_tr) at (5,1.8);
    \fill [fill=black!15!white] (bb_ll) rectangle (bb_tr);
    \draw [thick,name path=A] (bar_ll) to[out=0,in=180] (0.5,1) to[out=0,in=180] (2.5,1.8) to[out=0,in=180] (4.5,1) to[out=0,in=180] (bar_lr);
    \draw [thick,name path=B] (bar_tl) to[out=0,in=180] (0.5,1.8) to[out=0,in=180] (2.5,2.6) to[out=0,in=180] (4.5,1.8) to[out=0,in=180] (bar_tr);
    \tikzfillbetween[of=A and B] {black!35!blue, opacity=0.5};
    \node[label=left:\(\widehat{\Omega}\)] at (4.5,3.1) {};
    \node[label=left:\(\Omega\)] at (1.15,1.25) {};
    \node[label=right:\(\partial\Omega_{N_{\text{fixed}}}\)] at (4.8,1.5) {};
    \node[label=right:\(\partial\Omega_{N_{\text{free}}}\)] at (2.5,1.2) {};
    \node[label=left:\(\partial\Omega_{D}\)] at (0.2,1.5) {};
    \draw [thick] (bar_ll) -- (bar_tl) ;
    \draw [thick] (bar_lr) -- (bar_tr) ;
    \draw [-latex] (1.4,2.13) -- (1.1,2.65) node [label={[xshift=8pt,yshift=-12pt]\(\hat{n}\)}] {} ;
  \end{tikzpicture}
  \caption{Illustration of \(\Omega\) and its boundary components.\label{fig:adm_shape}}
\end{figure}
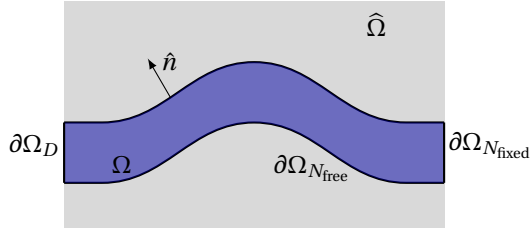

Since all feasible shapes have to coincide in $\Omega_D$ and in $\Omega_{N_{\text{fixed}}}$, it is natural to restrict the analysis to subsets of a sufficiently large bounded open set $\widehat{\Omega}\subset\R^p$ that satisfies $\partial\Omega_D\subseteq\partial\widehat{\Omega}$ and  
$\partial\Omega_{N_{\text{fixed}}}\subseteq\partial\widehat{\Omega}$ (see Figure~\ref{fig:adm_shape}). We additionally assume that $\widehat{\Omega}$ satisfies the \emph{cone property} for a given angle $\theta\in(0,\pi/2)$ and radii $r,l>0$, $r\leq l/2$, i.e., 
\begin{equation*}
\forall x\in\partial\widehat{\Omega}\; \exists \zeta_x\in\R^p,\|\zeta_x\|=1\;:\; 
y+C(\zeta_x,\theta,l) \subset\widehat{\Omega}\; \forall y\in B(x,r)\cap\widehat{\Omega}, 
\end{equation*}
where $C(\zeta_x,\theta,l):=\{c\in\R^p\,:\, \|c\|<l,\,c^{\top} \zeta_x>\|c\|\cos(\theta)\}$ is a truncated circular cone oriented along $\zeta_x$ with height $l$ and opening angle $2\theta$, and $B(x,r)\subset\R^p$ is an open ball of radius $r$ centered at $x$. 
Now the set of \emph{admissible shapes} $\Oad\subset {\mathcal P}(\R^p)$ can be defined as 
\begin{equation}\label{eq:Oad}
    \Oad := \left\{\Omega\subseteq\widehat{\Omega} \colon \partial\Omega_D\subseteq\partial\Omega,\, \partial\Omega_{N_{\text{fixed}}}\subseteq\partial\Omega,\; \widehat{\Omega} \text{ and } \Omega \text{ satisfy the cone property} \right\} .
\end{equation}


 Ceramic components behave according to the linear elasticity theory \cite{munz}. The state equation can be described as an elliptic partial differential equation, see, e.g., \cite{braess}. 
More precisely, we get the state equation which describes the reaction of the ceramic component to external forces as a partial differential equation:
\begin{equation} \label{stateequation}
\begin{array}{rcll}
  -\text{div}(\sigma(u(x))) & = & \tilde{f}(x) &  \text{for} \; x\in\Omega \\
  u(x) & = & 0 &  \text{for} \; x\in\partial\Omega_D    \\
 \sigma(u(x))\hat{n}(x) & = & \tilde{g}(x) &  \text{for} \; x\in\partial\Omega_{N_{\text{fixed}}} \\
  \sigma(u(x))\hat{n}(x) & = & 0 & \text{for} \; x\in\partial\Omega_{N_{\text{free}}} 
\end{array}
\end{equation} 
Here, $\hat{n}(x)$ is the outward pointing normal at $x\in\partial\Omega$, which is defined almost everywhere on $\partial \Omega$ given that $\partial \Omega$ is piecewise differentiable. Furthermore, let $\tilde{f}\in L^2(\Omega , \R^p)$ be the volume forces and $\tilde{g} \in L^2(\partial\Omega_{N_{\text{fixed}}} , \R^p)$ the forces acting on the surface $\partial \Omega_{N_{\text{fixed}}}$, e.g.,  
the tensile load. The displacement caused by the acting forces is given by $u
\in H^1(\Omega, \R^p)$, where $H^1(\Omega, \R^p)$ is the Sobolov space of 
$L^2(\Omega , \R^p)$-functions with weak derivatives in $L^2(\Omega , \R^{p\times p})$. The linear strain tensor $\varepsilon\in L^2(\Omega, \R^{p\times p})$ is given by $\varepsilon(u(x)):=\frac{1}{2}(D u(x) + (D u(x))^{\top})$,  where $D u$ is the Jacobi matrix of $u$. It follows for the stress tensor $\sigma\in L^2(\Omega, \R^{p\times p})$ that $\sigma(u(x))=\lambda\, \text{tr}(\varepsilon(u(x)))I+ 2\mu\varepsilon(u(x))$, where $\lambda, \mu >0$ are the Lam\'e constants derived from Young's modulus $E$ and Poisson´s ratio $\nu$ as $\lambda=\frac{\nu E}{(1+\nu)(1-2\nu)}$ and $\mu=\frac{E}{2(1+\nu)}$. 

From a numerical perspective, a variational formulation  of the state equation \eqref{stateequation} is usually preferred, see, e.g., \cite{PaperHahn,bolten}. This still guarantees a unique weak solution $u$, see \cite{duran}. 
Thus, $u$ is uniquely defined by the shape $\Omega$ \cite{duran}, and we will equivalently write $\sigma(Du(x)):=\sigma(u(x))$ for $x\in\Omega$ to highlight that $\sigma$ depends on the Jacobi matrix of $u$.

 
\subsection{Probability of Failure}\label{subsec:PoF}
The primary objective function, the mechanical integrity of the ceramic component, is modelled based on the probability of failure analogous to \cite{PaperHahn,bolten,bruecknerfoit,weibull}.  For the sake of completeness this is briefly summarized in the following.

We want to optimize the reliability of a ceramic body $\Omega$, i.e., its survival probability, by minimizing its probability of failure under tensile load. 
In that sense failure means that the ceramic body breaks under the tensile load due to cracks. Such cracks occur as a result of small faults in the material caused by the sintering process. 
To understand the mechanics of cracks, 
three types of crack opening are considered, see \cite{gross} and Figure~\ref{fig:modes} for an illustration. They are referred to as \emph{Modes I, II} and \emph{III}, respectively, and relate to different loads. Note that in the two-dimensional case,  only Modes I and II can occur. 
We refer to \cite{gross} for a detailed introduction into this topic.

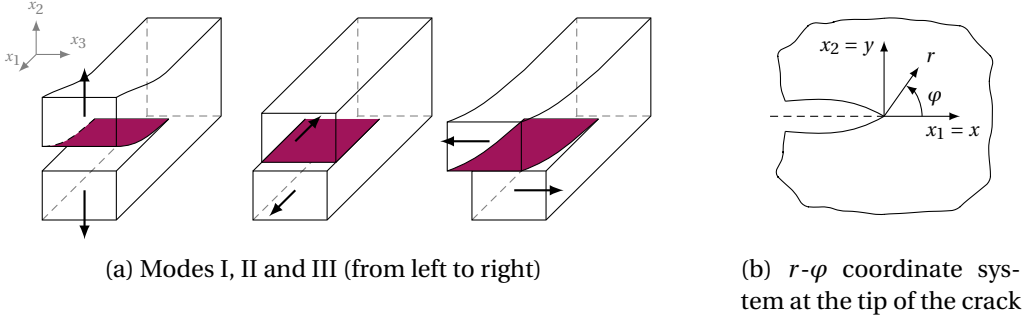
\begin{figure}[htb]
  \definecolor{fc}{RGB}{160,20,90}
  \centering
	\subfloat[Modes I, II and III (from left to right)\label{fig:modes}]{
\begin{tikzpicture}[%
x={(0.65cm,0cm)},
y={(0cm,0.65cm)},
z={({0.5*cos(45)},{0.5*sin(45)})}]
  \def\b{1.5}
  \def\h{1}
  \def\hh{2}
  \def\t{3}
  \def\tt{6}
  \def\lift{0.5}
\coordinate (A) at (0,0,0); 
\coordinate (B) at (\b,0,0) ;
\coordinate (C) at (\b,\h,0); 
\coordinate (D) at (0,\h,0); 
\coordinate (E) at (0,0,\tt); 
\coordinate (F) at (\b,0,\tt); 
\coordinate (G) at (\b,\h,\t); 
\coordinate (H) at (0,\h,\t);
\coordinate (A') at (0,\h+\lift,0); 
\coordinate (B') at (\b,\h+\lift,0) ;
\coordinate (C') at (\b,\hh+\lift,0); 
\coordinate (D') at (0,\hh+\lift,0);  
\coordinate (G') at (\b,\hh,\tt); 
\coordinate (H') at (0,\hh,\tt);
\coordinate (D'H') at (0,\hh,\t);
\coordinate (C'G') at (\b,\hh,\t);
\coordinate (anker1) at (0.5*\b,0.5*\h,0.1*\t);
\coordinate (anker2) at (0.5*\b,1.5*\h+\lift,0.1*\t);
\draw[gray,-latex] (-0.3,3.2,0.5) -- (-0.3,3.2,-0.5) node [label={[xshift=-2pt, yshift=-5pt]\tiny\(x_1\)}]{};
\draw[gray,-latex] (-0.3,3.2,0.5) -- (-0.3,3.9,0.5) node [label={[xshift=0pt, yshift=-5pt]\tiny\(x_2\)}]{};
\draw[gray,-latex] (-0.3,3.2,0.5) -- (0.4,3.2,0.5) node [label={[xshift=3pt, yshift=-5pt]\tiny\(x_3\)}]{};
\draw[densely dashed,gray] (A) -- (E) -- (F);
\draw[densely dashed,gray] (E) -- (H');
\draw (D) -- (H) -- (G);
\fill[fc] (A') -- (B') to[out=0,in=-140] (G) -- (H) to[in=0,out=-140] cycle;
\draw[] (G) -- (C) -- (D) -- (A) -- (B) -- (F) -- (G') -- (H');
\draw[] (B) -- (C);
\draw[] (A') -- (B') -- (C') -- (D') -- cycle;
\draw[] (B') to[out=0,in=-140] (G);
\draw[densely dashed] (A') to[out=0,in=-140] (H);
\draw[] (D') to[out=30,in=-135] (D'H') -- (H');
\draw[] (C') to[out=30,in=-135] (C'G') -- (G');
\draw[thick,-latex] (anker1) -- ++(0,-2*\lift,0);
\draw[thick,-latex] (anker2) -- ++(0,2*\lift,0);
\useasboundingbox (-0.5,-0.5) rectangle (3.5,4);
\end{tikzpicture}
\begin{tikzpicture}[%
x={(0.65cm,0cm)},
y={(0cm,0.65cm)},
z={({0.5*cos(45)},{0.5*sin(45)})},
]
  \def\b{1.5}
  \def\h{1}
  \def\hh{2}
  \def\t{3}
  \def\tt{6}
  \def\push{0.5}
\coordinate (A) at (0,0,0); 
\coordinate (B) at (\b,0,0) ;
\coordinate (C) at (\b,\h,0); 
\coordinate (D) at (0,\h,0); 
\coordinate (E) at (0,0,\tt); 
\coordinate (F) at (\b,0,\tt); 
\coordinate (G) at (\b,\h,\t); 
\coordinate (H) at (0,\h,\t);
\coordinate (A') at (0,\h,0+\push); 
\coordinate (B') at (\b,\h,0+\push) ;
\coordinate (C') at (\b,\hh,0+\push); 
\coordinate (D') at (0,\hh,0+\push);  
\coordinate (G') at (\b,\hh,\tt); 
\coordinate (H') at (0,\hh,\tt);
\coordinate (anker1) at (0.5*\b,0.5*\h,0.1*\t);
\coordinate (anker2) at (0.5*\b,1.5*\h,0.1*\t);
\draw[densely dashed,gray] (A) -- (E) -- (F);
\draw[densely dashed,gray] (E) -- (H');
\draw (D) -- (H);
\draw[] (G) -- (C) -- (D) -- (A) -- (B) -- (F) -- (G') -- (H');
\draw[] (B) -- (C);
\fill[fc] (A') -- (B') -- (G) -- (H) -- cycle;
\draw[] (A') -- (B') -- (C') -- (D') -- cycle;
\draw[] (B') -- (G);
\draw[] (A') -- (H);
\draw[] (D') -- (H');
\draw[] (C') -- (G');
\draw[thick,-latex] (anker1) -- ++(0,0,-3*\push);
\draw[thick,-latex] (anker2) -- ++(0,0,3*\push);
\useasboundingbox (-0.5,-0.5) rectangle (3.5,4);
\end{tikzpicture}
\begin{tikzpicture}[%
x={(0.65cm,0cm)},
y={(0cm,0.65cm)},
z={({0.5*cos(45)},{0.5*sin(45)})}]
  \def\b{1.5}
  \def\h{1}
  \def\hh{2}
  \def\t{3}
  \def\tt{6}
  \def\shift{0.5}
\coordinate (A) at (0,0,0); 
\coordinate (B) at (\b,0,0) ;
\coordinate (C) at (\b,\h,0); 
\coordinate (D) at (0,\h,0); 
\coordinate (E) at (0,0,\tt); 
\coordinate (F) at (\b,0,\tt); 
\coordinate (G) at (\b,\h,\t); 
\coordinate (H) at (0,\h,\t);
\coordinate (A') at (0-\shift,\h,0); 
\coordinate (B') at (\b-\shift,\h,0) ;
\coordinate (C') at (\b-\shift,\hh,0); 
\coordinate (D') at (0-\shift,\hh,0);  
\coordinate (G') at (\b,\hh,\tt); 
\coordinate (H') at (0,\hh,\tt);
\coordinate (D'H') at (0,\hh,\t);
\coordinate (C'G') at (\b,\hh,\t);
\coordinate (anker1) at (0.5*\b,0.5*\h,0.1*\t);
\coordinate (anker2) at (0.5*\b-\shift,1.5*\h,0.1*\t);
\draw[densely dashed,gray] (A) -- (E) -- (F);
\draw[densely dashed,gray] (E) -- (H');
\draw[] (G) -- (C) -- (D) -- (A) -- (B) -- (F) -- (G') -- (H');
\draw[] (B) -- (C);
\draw (D) -- (H) -- (G);
\fill[fc] (A') -- (B') to[out=20,in=-140] (G) -- (H) to[in=20,out=-140] cycle;
\draw[] (A') -- (B') -- (C') -- (D') -- cycle;
\draw[] (B') to[out=20,in=-140] (G);
\draw[] (A') to[out=20,in=-140] (H);
\draw[] (D') to[out=20,in=-140] (D'H') -- (H');
\draw[] (C') to[out=20,in=-140] (C'G') -- (G');
\draw[thick,-latex] (anker1) -- ++(2*\shift,0,0);
\draw[thick,-latex] (anker2) -- ++(-2*\shift,0,0);
\useasboundingbox (0,-0.5) rectangle (3.5,4);
\end{tikzpicture}
}
\hspace*{1cm}
\subfloat[$r$-$\varphi$ coordinate system at the tip of the crack\label{fig:process} ]{
\begin{tikzpicture}[scale=1,decoration={random steps,segment length=5pt,amplitude=1pt}]
  \coordinate (O1) at (-1.3,0.2);
  \coordinate (O2) at (-1.3,-0.2);
  
  \draw (O1) to[out=5,in=150] (0,0);
  \draw (O2) to[out=-5,in=210] (0,0);
  \draw[densely dashed] (-1.5,0) -- (0,0);
  \draw[-latex] (0,0) -- (1,0) node[pos=0.9,below] {\scriptsize\(x_1=x\)};
  \draw[-latex] (0,0) -- (0,1) node[pos=0.9,left] {\scriptsize\(x_2=y\)};
  \draw[-latex] (0,0) -- (0.4589,0.6553) node[pos=0.9,above right] {\scriptsize\(r\)};
  \draw[-latex] (0.5,0) arc (0:55:0.5) node[midway,right] {\scriptsize\(\varphi\)};
  \draw[decorate,rounded corners=2] (O1) -- (-1.4,0.8) -- (-1,1.2) -- (0,1.3) -- (1,1.2) -- (1.4,0.8) -- (1.4,-0.8) -- (1,-1.2) -- (0,-1.3) -- (-1,-1.2) -- (-1.4,-0.8) -- (O2);
  \useasboundingbox (-1.8,-1.7) rectangle (1.8,1.5);
\end{tikzpicture}}
\caption{Crack opening modes and two-dimensional model for the crack-tip field according to \cite{PaperHahn,bolten,gross}}
\end{figure}

The stresses and strains close to a crack are represented by the \emph{crack-tip field} which depends on the respective crack opening modes.  It is described locally by a two-dimensional model, see Figure~\ref{fig:process} for an illustration.  
With $K_{\text{I}}, K_{\text{II}}$ and $K_{\text{III}}$ being the \emph{stress-intensity factors} (also called \emph{$K$-factors}) corresponding to Modes I, II, and III, respectively, one can describe  the crack-tip field $\sigma$ locally according to linear fracture mechanics as 
\begin{align}
\sigma(x)=\sigma(r,\phi) 
= \frac{1}{\sqrt{2\pi r}} \Bigl\{K_{\text{I}}\tilde{\sigma}^\text{I}(\phi )+K_{\text{II}}\tilde{\sigma}^{\text{II}}(\phi )+K_{\text{III}}\tilde{\sigma}^{\text{III}}(\phi ) \Bigr\} + R(r,\phi).
\label{sigma_K_formula}
\end{align}
Here, $r$ is the distance to the crack tip, 
and $\phi$ the angle w.r.t.\ the $x_1$-axis (aligned with the crack plane), see Figure~\ref{fig:process}. The functions $\tilde{\sigma}^{\text{I,II,III}}(\phi)$ are known functions of the angle $\phi$, see again \cite{gross}, and $R(r,\phi)$ is a regular function of the considered position in $x\in\Omega$ that is independent of the crack. 
Note that in the two-dimensional case, Mode III is omitted from \eqref{sigma_K_formula} since it does not exist. Moreover, experimental evidence has shown that Mode~I, which relates to tensile and compressive load, is the most relevant for the failure of ceramic structures \cite{bruecknerfoit}, see \cite{gross} for approaches for multi-mode failure. We will thus focus on $K_{\text{I}}$ in the following as the driving parameter for crack development under tensile load. 

In order to evaluate $K_{\text{I}}$ analogous to \cite{bolten}, 
we adopt the concept of equivalent circular discs to represent different crack shapes and crack sizes, and hence assume that the cracks are \emph{penny shaped}. Then a particular crack can be identified by its configuration 
\begin{align*}
(x,a,n )\in \mathcal{C}:= {\Omega}\times (0,\infty )\times S^{p-1} ,
\end{align*}
where $x \in {\Omega}$ is its location, $a \in (0, \infty)$ its radius, and $n \in S^{p-1}$ its orientation ($S^{p-1}$ denotes the unit sphere in $\R^p$). $\mathcal{C}$ is called the \emph{crack configuration space}.
Given a crack $(x,a,n)\in\mathcal{C}$, $K_{\text{I}}$ can be computed as a function of the radius $a$ and of the tensile load  $\sigma_n(Du(x))$ in the normal direction $n$ of the stress plane at the crack location $x$ as
\begin{align}
K_{\text{I}} = K_{\text{I}}(a,\sigma_n(Du(x)))= \frac{2}{\pi}\sigma_n(Du(x))\sqrt{\pi a},
\label{Ki}
\end{align}
see, e.g., Table~4.1 in \cite{gross}.  
Following \cite{bolten} we set 
\begin{align*}
\sigma_n(Du(x)):=
\max \{ n^{\top}  \,\sigma (Du(x))\, n\, ,\, 0 \}.
\end{align*}
Note that negative values of $\sigma_n(x)$ correspond to compressive loads which can be ignored in the analysis of crack development, see Figure~\ref{fig:modes} above.

A crack $(x,a,n)\in\mathcal{C}$ becomes critical, i.e., a fracture occurs and the material fails, if $K_{\text{I}}$ exceeds a material-specific critical value $K_{\text{I} c}$ (the \emph{ultimate tensile strength} of the material).  
Note that \eqref{Ki} implies that all cracks with radius 
\begin{align}
a > a_c := \frac{\pi}{4}\left(\frac{K_{\text{I} c}}{\sigma_n(Du(x))}\right)^2 
\label{eq:ac}
\end{align}
are critical.
We denote the set of \emph{critical configurations} by  $$A_c:=A_c(\Omega,Du)=\{(x,a,n)\in\mathcal{C}\; :\; K_{\text{I}}(a,\sigma_n(Du(x)))>K_{\text{I} c}\}$$ 
and want to minimize the probability of finding a crack with configuration in  $A_c$. 

Following \cite{PaperHahn,bolten}, 
we assume that the parameters $(x,a,n)$ are random (i.e., they are not deterministically given by the sintering process), that the cracks are statistically homogeneously distributed in $\Omega$, and that their orientations are isotropic.
Let $A\subseteq\mathcal{C}$ be a measurable subset of the configuration space. Then under quite general assumtions the random number $N(A)$ of cracks in $A$ is Poisson distributed (see 
\cite{kallenberg,watanabe}), and hence $N(A)$ is a Poisson point process.
It follows that $P(N(A)=k)=e^{-\upsilon (A)}\frac{\upsilon (A)^k}{k!}\sim Po(\upsilon (A))$, 
where $\upsilon$ 
is the (Radon) \emph{intensity measure} of the process. 
Recall that a component fails if $N(A_c)>0$. 
Given a displacement field $u \in H^1(\Omega, \R^p)$, we can now write the survival probability of the component $\Omega$ as
\begin{equation*}
p_s(\Omega|Du)=P(N(A_c(\Omega ,Du))=0)=\exp\{-\upsilon (A_c(\Omega, Du))\}.
\end{equation*}
Hence, to maximize the survival probability of a component $\Omega$ we need to minimize the intensity measure $\upsilon$. 
Since only cracks $(x,a,n)$ with radius $a>a_c$ need to be considered (c.f.\ 
\eqref{eq:ac} above), \cite{PaperHahn,bolten} determine the intensity measure as
\begin{align*}
\upsilon (A_c(\Omega, Du))=\frac{\Gamma(\frac{p}{2})}{2\pi^{\frac{p}{2}}}\int\limits_{\Omega}\int\limits_{S^{p-1}}\int\limits_{a_c}^{\infty}\text{d}\upsilon_a(a)\, \text{d}n\, \text{d}x
\end{align*}
with $\text{d}x$ the Lebesgue measure on $\mathbb{R}^p$, $\text{d}n$ the surface measure on $S^{p-1}$, and $\text{d}\upsilon_a (a)= c\cdot a^{-\tilde{m}}\text{d}a$ being a positive Radon measure modelling the occurrence of cracks of radius $a$ in $\Omega$ 
($c>0$ and $\tilde{m}\geq\frac{3}{2}$ are positive constants).
Note that for $p=3$ the $\Gamma$-function takes the value $\Gamma(\frac{3}{2})=\frac{\sqrt{\pi}}{2}$ and for $p=2$ we obtain $\Gamma(1)=1$. 
With $m:=2(\tilde{m}-1)\geq 1$ and using again \eqref{eq:ac} 
the inner integral can be evaluated, yielding
\begin{align*}
\upsilon (A_c(\Omega,Du))=\frac{\Gamma(\frac{p}{2})}{2\pi^{\frac{p}{2}}}\int\limits_{\Omega}\int\limits_{S^{p-1}}\left(\frac{\sigma_n(Du(x))}{\sigma_0}\right)^m \text{d}n\,\text{d}x,
\end{align*}
where $\sigma_0$ is an appropriately chosen  positive constant. As highlighted in \cite{PaperHahn,bolten}, this is in accordance with the statistical model introduced by Weibull \cite{weibull}. In this context, the parameter $m$ is referred to as \textit{Weibull module} and typically assumes values between $5$ and $25$.

Summarizing the discussion above, we define our primary objective function
$f_1:\Oad\rightarrow\R$ as
\begin{align}\label{eq:obj1}
f_1(\Omega) 
:=\upsilon (A_c(\Omega,Du))
\end{align}
and refer to it as \emph{intensity measure}, modelling the probability of failure (PoF) of the component $\Omega$. 
Recall that $u(\Omega)$ is uniquely defined by $\Omega$ and thus $f_1(\Omega)$ is completely defined by the shape $\Omega$ (given fixed boundary conditions $\tilde{f},\tilde{g}$).

\subsection{Material Consumption}\label{subsec:volume} \label{subsubsec:MaterialConsumption} 
Improving the intensity measure $f_1$ of a ceramic component (and hence its PoF) usually comes at the price of an increased material consumption, which is directly correlated with the cost of the component. In order to avoid excessively expensive solutions, classical approaches thus set a predetermined bound on the allowable volume of the shape $\Omega$ (see, e.g.,  \cite{PaperHahn,bolten}). We follow a more general approach in this manuscript and interpret the volume (and hence the cost) of the component as an equitable second objective function. This facilitates, in particular, the analysis of the trade-off between these two criteria and supports the engineer in finding a preferable design. 
We thus define $f_2:\Oad\rightarrow\R$ as the \emph{volume} of a shape $\Omega\in\Oad$ given by
\begin{equation}\label{obfun_vol}
f_2(\Omega )  := \int\limits_{\Omega}\text{d}x.
\end{equation}

\subsection{Biobjective Optimization} \label{subsec:bioptalg}

When multiple conflicting goals are relevant in an optimization problem, a common approach is to use a weighted sum of the individual objectives as an overall objective function and then resort to classical optimization algorithms. The advantages and also the shortcomings of this so-called \emph{weighted sum scalarization} are discussed in the following section, see also \cite{ehrgott}. Particularly when choosing fixed weights, this method is of limited applicability. While fixed weights may represent the preferences of one decision maker, another decision maker may have other preferences, i.e., other weights. Moreover, the objective ranges and the scales of the objectives may be very different or even incomparable, which generally leads to numerical difficulties. 

Another common approach to handle multiple conflicting goals is to select one ``most important'' objective function to minimize, e.g., the probability of failure, and set upper bounds on the acceptable objective function values of the other objective functions. In our case this would imply a constraint on the allowable material consumption, see, e.g., \cite{PaperHahn,bolten}. This approach is referred to as \emph{$\varepsilon$-constraint scalarization}, see again \cite{ehrgott} for a general discussion of this topic. In addition to the numerical difficulties that may arise from adding potentially complicating constraints to the problem formulation, this approach has similar drawbacks as the weighted sum scalarization: The selection of meaningful upper bound values may be difficult, and trade-off information is ignored.

A more general approach is to formulate a multiobjective optimization problem, and hence to compute a set of relevant solution alternatives rather than one  single ``optimal'' solution. 
By providing a set of solution alternatives the decision maker can not only choose a solution that aligns the most with his preferences, but he can also inspect the trade-off between alternative solutions and can adjust his preferences accordingly. A decision maker may, for example, prefer reliability over volume, but looking into the trade-off between solution alternatives there may be a solution that is some small percentage worse w.r.t.\ the reliability while it is a lot better regarding the volume. This may lead to a re-evaluation of the decision maker's preferences.

With our two objective functions ``intensity measure'' ($f_1$, modeling the PoF) and ``volume'' ($f_2$), the following \emph{biobjective shape optimization problem} arises:
\begin{equation}
\begin{split}
\min_{\Omega\in\Oad} & ~f(\Omega):=(f_1(\Omega),f_2(\Omega))\\
\text{s.t. } & u \in H^1(\Omega, \R^p) \text{ solves the state equation } (\ref{stateequation}),
\end{split}\label{ceramicMOP}
\end{equation}
where $f_1$ and $f_2$ are defined according to Sections~\ref{subsec:PoF} and \ref{subsubsec:MaterialConsumption} above.  Note that only $f_1$ depends on the displacement field $u(\Omega)$.

We call $f=(f_1,f_2):{\Oad}\longrightarrow\R^2$ the \textit{biobjective function vector} and  $\R^2$  the \textit{objective space}.  
Let $Z:=f(\Oad)\subset\R^2$ denote the set of all \emph{feasible outcome vectors} in the objective space, i.e., the set of all outcome vectors that are images of admissible shapes $\Omega\in{\Oad}$. In contrast to single objective optimization we have to define optimality in the presence of two objectives, since there is no natural order on $\R^2$. For two shapes $\Omega_1,\Omega_2 \in {\Oad}$, let $z^1=f(\Omega_1)$ and $z^2=f(\Omega_2)$
be the respective outcome vectors in $Z$. We write
$$\begin{array}{ccl}
z^1 \leqq z^2 & \iff & z^1_j\leq z^2_j,\; j=1,2\\[0.2cm]
z^1 \leqslant z^2 & \iff & z^1\leqq z^2 \text{ and } z^1\neq z^2\\[0.2cm]
z^1<z^2 & \iff & z^1_j < z^2_j,\; j=1,2.
\end{array}$$
Note that $z^1\leqslant z^2$ implies that $z^1_j\leq z^2_j$ for $j=1,2$ with at least one strict inequality. 
We use the notation 
$$\R^2_{\geqslant}:=\{z\in\R^2\,:\, z\geqslant (0,0)^{\top}\} \quad \text{and} \quad \bar{z}+\R^2_{\geqslant}:=\{z\in\R^2\,:\, z\geqslant \bar{z}\} \quad\text{for~} \bar{z}\in\R^2.$$ 
The notations $\R^2_{\geqq}$, $\R^2_{>}$, $\R^2_{\leqslant}$, $\R^2_{\leqq}$ and $\R^2_{<}$ are used accordingly.

We say that $z^1$ \emph{dominates} $z^2$ if and only if $z^1 \leqslant z^2$, i.e., if and only if $z^1\in z^2+\R^2_{\leqslant}$.  An outcome vector $\bar{z}\in Z$ is called \emph{nondominated} if there is no other outcome vector $z\in Z$ such that $z\leqslant\bar{z}$.  
Accordingly, an admissible shape ${\Omega}_P\in {\Oad}$ is called \emph{Pareto optimal} or \emph{efficient}, if there is no other admissible  shape $\Omega\in {\Oad}$ such that $f(\Omega)\leqslant f({\Omega}_P)$. 
We are mainly interested in Pareto optimal shapes since these are precisely those shapes that can not be improved in one objective without deterioration in the other objective. The set of all Pareto optimal shapes is called the \emph{Pareto front} and denoted by $\Oad_P$. Similarly, the set of all nondominated outcome vectors $Z_N:=f(\Oad_P)$ is referred to as the \emph{nondominated front} in the objective space.

As in single-objective optimization, one often has to resort to local minima if the underlying optimization problem is nonconvex (and difficult). In the biobjective setting, an admissible shape ${\Omega}_{\ell P}\in {\Oad}$ is called \emph{locally Pareto optimal} or \emph{locally efficient}, if there is a neighborhood $\mathcal{N}\subseteq\Oad$ of $\Omega_{\ell P}$ such that there is no other admissible shape $\Omega\in\mathcal{N}$ with  $f(\Omega)\leqslant f({\Omega}_{\ell P})$.

\subsection{Existence of Pareto Optimal Shapes}\label{subsec:exist}

In order to prove the existence of Pareto optimal shapes, we consider the weighted sum scalarization of problem \eqref{ceramicMOP} in which, given a weight $\omega\in(0,1)$, the two objective functions are combined into one single weighted sum objective:
\begin{equation}
\begin{split}
\min_{\Omega\in\Oad} & f_{\omega}(\Omega) := \omega f_1(\Omega) + (1-\omega) f_2(\Omega)\\
\text{s.t. } &  u \in H^1(\Omega, \R^p) \text{ solves the state equation } (\ref{stateequation}).
\end{split}\label{eq:shapeweightedsum}
\end{equation}
It is a well-known fact that every optimal solution of problem \eqref{eq:shapeweightedsum} is Pareto optimal for problem \eqref{ceramicMOP}, see, e.g., \cite{ehrgott}.

\begin{theorem}
If the crack size measure has the non decreasing stress hazard property (see \cite{bolten} for a formal definition), then the set $\Oad_P$ is non-empty.
\end{theorem}

\begin{proof}
Suppose that $\omega\in(0,1)$ is chosen arbitrarily, but fixed. Then the weighted sum objective can be evaluated as
\begin{eqnarray*}
f_{\omega}(\Omega)
&\!=\!& \omega\left(\frac{\Gamma(\frac{p}{2})}{2\pi^{\frac{p}{2}}}\int\limits_{\Omega}\int\limits_{S^{p-1}}\left(\frac{\sigma_n(Du(x))}{\sigma_0}\right)^m \text{d}n \, \text{d}x\right) + (1\!-\!\omega)\int\limits_{\Omega}\text{d}x\\
&\!=\!& \omega\,\frac{\Gamma(\frac{p}{2})}{2\pi^{\frac{p}{2}}}\int\limits_{\Omega} \int\limits_{S^{p-1}}\left(\frac{\sigma_n(Du(x))}{\sigma_0}\right)^m \text{d}n \; + \; \underbrace{\frac{2\pi^{\frac{p}{2}}(1\!-\!\omega)}{\Gamma(\frac{p}{2})\omega}}_{\text{constant}} \text{d}x .
\end{eqnarray*}
Thus, the incorporation of $f_2$ into the scalarized objective function corresponds to the addition of a constant term in the shape integral of $f_1$. This does not affect the convergence analysis of \cite{bolten}, which is based on convexity of the integrand in $Du$, see \cite{chenais,fuji}. We can conclude that the weighted sum scalarization has an optimal solution for every $\omega\in(0,1)$. Since every such solution is Pareto optimal for \eqref{ceramicMOP}, the result follows.
\end{proof}

\section{Numerical Implementation}\label{sec:Implementation}
To actually compute locally Pareto optimal shapes, we adopt the finite element discretization implemented in \cite{PaperHahn} for two-dimensional instances (i.e., $p=2$). In this implementation, 
the shapes $\Omega\in\Oad$, the state equation \eqref{stateequation}, the objective functions $f_1$ and $f_2$ and their gradients are discretized. Standard Lagrangian finite elements are used for the discretization of the state equation~\eqref{stateequation},  and all integrals are calculated using numerical quadrature. The discretized shape gradients are obtained by an adjoint approach to reduce computational costs. We refer to \cite{PaperHahn} for a detailed description.

\subsection{Geometry Definition and Finite Element Mesh}
\label{subsec:Geometry}
The two-dimensional shapes $\Omega \in \Oad\subset {\mathcal P}(\R^2)$ are discretized by an $n_x\times n_y$ mesh \(X:=X^{\Omega}=(X^{\Omega}_{ij})_{n_x\times n_y}\) (we write \(X_{ij}:=X_{ij}^{\Omega}\in\R^2\) for short) using tetrahedrons, 
with $n_x,n_y \in \mathbb{N}$ being the number of grid points in $x$ and $y$ direction, respectively. Given a shape $\Omega\in\Oad$ and its discretization $X$, the objective function values $f_1(X)$ and $f_2(X)$ as well as their gradients $\nabla f_1(X)$ and $\nabla f_2(X)$ are computed using the implementation of \cite{PaperHahn}.

For the optimization process, we fix the $x$-component of all grid points to equidistant values  $x_1,\dots,x_{n_x}$, and we only consider the $y$-components of those grid points that define the boundary of the shape to avoid deformation of the inner mesh structure. Note that this reformulation reduces the number of optimization variables from $2 n_x n_y$ to $2 n_x$.
As a consequence, feasible shapes can alternatively be represented by a \emph{shape parameter} $\varrho$ containing, for every relevant $x$-coordinate, the $y$-coordinate of the \emph{meanline} $\varrho^{\text{ml}}_i\in\R$ of the shape, and the \emph{thickness}  $\varrho^{\text{th}}_i \in \R_>$ of the shape, $i=1,\dots,n_x$. 
Given a feasible shape represented by $\varrho:=(\varrho^{\text{ml}},\varrho^{\text{th}})\in\R^{ 2n_x}$ with $\varrho^{\text{th}}\in \R_>^{n_x}$, an  associated mesh representation 
$X$ can be obtained using
\begin{equation}\label{generategrid}
X_{i,j}:=\left(x_i\; ,\; \varrho^{\text{ml}}_i+\frac{\varrho^{\text{th}}_i}{n_{y}-1}\left( j-\frac{n_{y}+1 }{2} \right) \right)\in\R^2,
\qquad i=1,\ldots,n_x, \; j=1,\ldots,n_y.
\end{equation}
%
To further reduce the computational burden and to obtain smoother shapes, the shape parameters
$\varrho^{\text{ml}}\in\R^{n_x}$ and $\varrho^{\text{th}}\in\R^{n_x}_>$ are modelled using \textit{B-splines}.  
Let $n_B \in  \mathbb{N}$, with $n_B < n_x$, be the number of B-spline basis functions, and let $\{\vartheta_j : \R\rightarrow \R_{\geq}, \; j=1, \dots, n_B  \}$ be a B-spline basis (see, e.g., \cite{nurbs}).  
Feasible shapes are then represented by B-spline coefficients $\gamma:=(\gamma^{\text{ml}}, \gamma^{\text{th}}) \in \R^{2 n_B}$.
The corresponding meanline and thickness values can be computed using the auxiliary functions
\begin{equation*}
\hat{\varrho}^{\text{ml}}(x):=\sum_{j=1}^{n_B} \gamma^{\text{ml}}_j \, \vartheta_j(x) \quad \text{and}\quad \hat{\varrho}^{\text{th}}(x):=\sum_{j=1}^{n_B} \gamma^{\text{th}}_j \, \vartheta_j(x), \qquad x\in \R.
\end{equation*}
These auxiliary meanline and thickness functions are then evaluated at the fixed $x$-coordinates of the gridpoints which yields 
\begin{equation}\label{eq:b-sp}
{\varrho}^{\text{ml}}_i:= \hat{\varrho}^{\text{ml}}(x_i) \quad \text{and} \quad 
{\varrho}^{\text{th}}_i:= \hat{\varrho}^{\text{th}}(x_i),\qquad i=1,\dots,n_x.
\end{equation}

Using the B-spline coefficients $\gamma=(\gamma^{\text{ml}}, \gamma^{\text{th}}) \in \R^{2 n_B}$
as optimization variables yields a further reduction of the number of variables to $2 n_B$. Moreover, the B-spline representation leads to an implicit regularization and smoothing of the represented shapes. In the following, we denote the set of \emph{feasible shape parametrizations} by
$\Gamma\subseteq\{(\gamma^{\text{ml}}, \gamma^{\text{th}})\in\R^{2n_B}\}$.

To evaluate the objective functions $f_j(\gamma)$ and their gradients $\nabla f_j(\gamma)={\partial f_j}/{\partial\gamma}$, $j=1,2$, w.r.t.\ the new parametrization of shapes based on B-spline parameters $\gamma$, while still using the implementation of \cite{PaperHahn},   
we compute an associated grid \(X\) using first \eqref{eq:b-sp} and then \eqref{generategrid}. 
While the resulting objective function values can be used immediately in the optimization process, the
gradients computed w.r.t.\ the grid $X$ need to be translated to the space of B-spline coefficients, i.e., 
\begin{equation}\label{eq:L2gradient}
 \frac{\partial f_j}{\partial\gamma^{\text{ml}}}=\frac{\partial f_j}{\partial X}\,\frac{\partial X}{\partial\varrho^{\text{ml}}}\,\frac{\partial\varrho^{\text{ml}}}{\partial\gamma^{\text{ml}}}\quad\text{and}\quad 
 \frac{\partial f_j}{\partial\gamma^{\text{th}}}=\frac{\partial f_j}{\partial X}\,\frac{\partial X}{\partial\varrho^{\text{th}}}\,\frac{\partial\varrho^{\text{th}}}{\partial\gamma^{\text{th}}},\qquad j=1,2.
\end{equation}

The numerical computation of gradients of $f_j$, $j=1,2$, w.r.t.\ a B-spline representation $\gamma$ of a feasible shape $\Omega$ is thus based on a two-step projection of $\gamma$ onto the original grid $X$. The thus computed gradients of $f_1$ (the intensity measure) were validated, using finite differences, at the sample shape shown in Figure~\ref{fig:TC1_Start}. 
The validation is based on a grid $(X_{ij})_{41\times 7}$, i.e., $n_x=41$ and $n_y=7$. Consequently, for the corresponding meanline and thickness representation we have $\varrho=(\varrho^{\text{ml}},\varrho^{\text{th}}) \in \R^{82}$, where $\varrho^{\text{th}} \in \R^{41}_>$.  Moreover, we used a B-spline basis with five basis functions, i.e., $n_B=5$ and $\gamma=(\gamma^{\text{ml}}, \gamma^{\text{th}}) \in \R^{10}$. We computed all ten partial derivatives w.r.t.\ $\gamma$ via the respective transformations to the grid representation and compared them with finite differences. 
The results of this comparison, i.e., the absolute values of the differences between computed derivatives and finite differences, are shown in Figure~\ref{fig:fd_g_m} and \ref{fig:fd_g_t} for the meanline and thickness parameters, respectively.  The figures indicate in all cases that, when the finite differences are evaluated for decreasing values of the increment $\varepsilon$, 
then they correspond well to the computed gradients. 

\begin{figure}[!htb]
	\begin{center}
		\subfloat[Validation of ${\partial f_1}/{\partial\gamma_i^{\text{ml}}}$\label{fig:fd_g_m}]
		    {\includegraphics[width=0.4\textwidth]{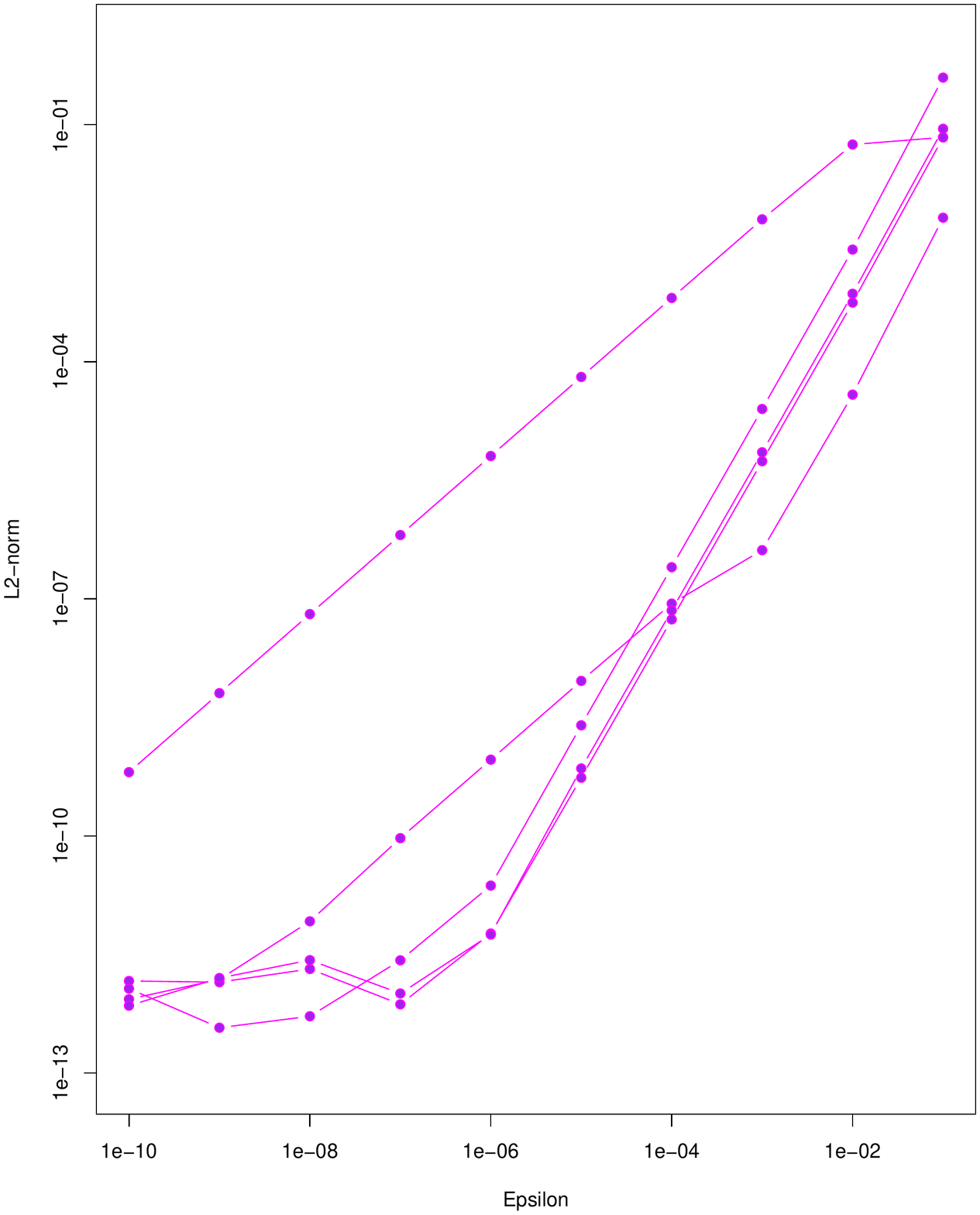}}
		\hspace{0.5cm }
		\subfloat[Validation of ${\partial f_1}/{\partial\gamma_i^{\text{th}}}$\label{fig:fd_g_t}]
		    {\includegraphics[width=0.4\textwidth]{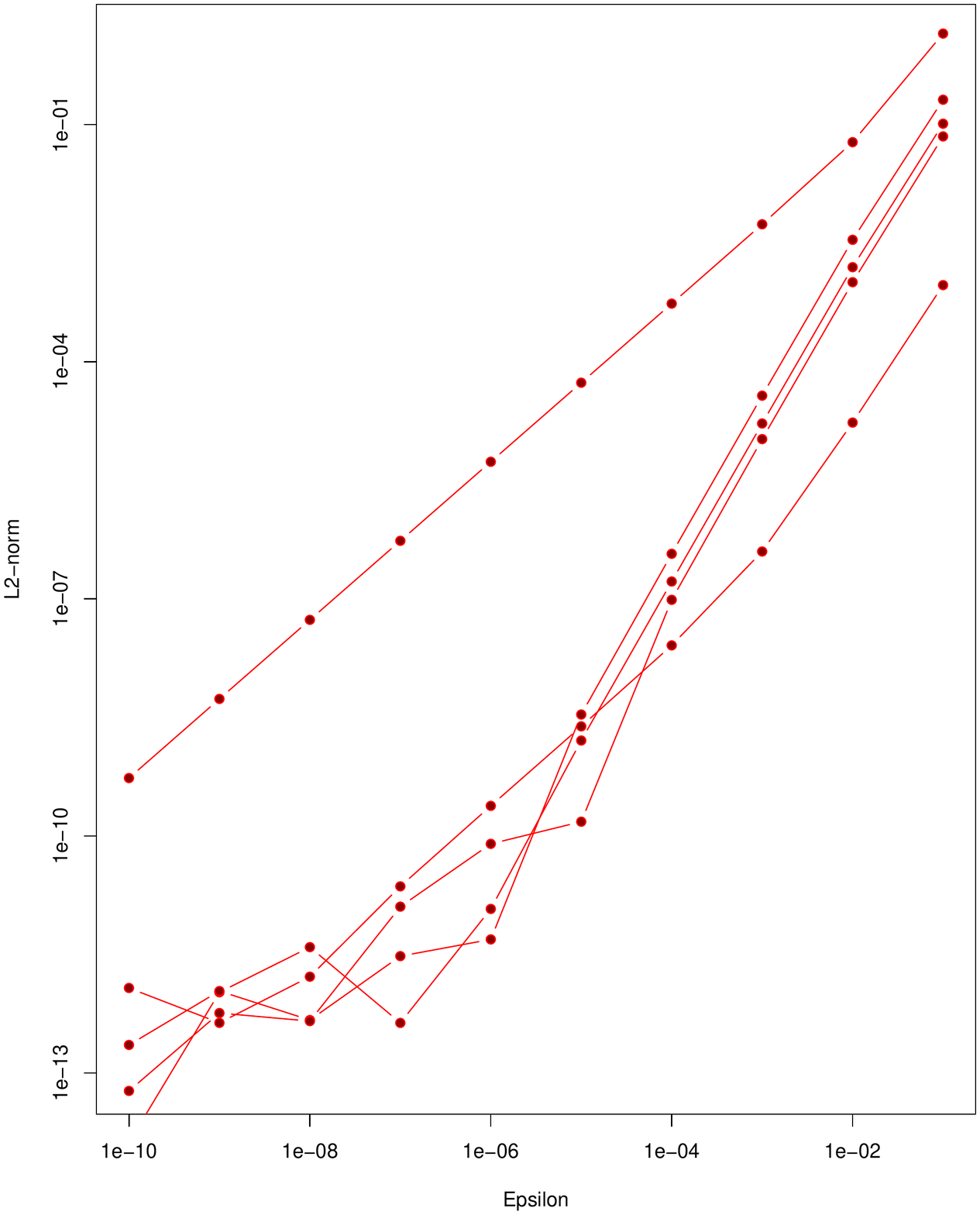}}
		\caption{Validation of gradients computed according to \eqref{eq:L2gradient} using finite differences. On the $x$-axis: increment $\varepsilon$ used for the finite difference evaluation; on the $y$-axis: absolute deviation between ${\partial f_1}/{\partial\gamma_i^{\text{ml},\text{th}}}$ computed according to \eqref{eq:L2gradient} and the corresponding finite difference, $i=1,\dots,5$, for meanline (left) and thickness (right).  \label{fig:findiff}}
	\end{center}
\end{figure}

\subsection{Pareto Critical Solutions}\label{subsec:Paretocritical}
Given the parametrization of admissible shapes described in Section~\ref{subsec:Geometry}, the biobjective optimization problem \eqref{ceramicMOP} can now be restated as
\begin{equation}
\begin{split}
\min_{\gamma\in\Gamma} & (f_1(\gamma),f_2(\gamma))\\
\text{s.t. } & u(X(\gamma)) \text{ solves the discretized state equation } (\ref{stateequation}).
\end{split}\label{discretizedMOP}
\end{equation}
Recall that only $f_1$ depends on the displacement field $u(X)$. 

Since derivative information is available, necessary optimality conditions can be formulated that generalize the concept of critical points from single-objective optimization.  Towards this end, we omit the constraints implied by the parametric representation of admissible shapes to keep the exposition simple. All constraints will be handled implicitly in the numerical tests described in Section~\ref{sec:CaseStudies} below. 
Assuming that both objective functions are continuously differentiable a necessary condition for a solution $\gamma \in \R^{2n_B}$ to be locally Pareto optimal is that
\begin{align}
\Bigl\{d \in  \mathbb{R}^{2n_B}\; : \; \nabla f_j(\gamma)^\top\,d < 0,\; j=1,2 \Bigr \}= \varnothing,
\label{Pareto_crit}
\end{align}
i.e., there does not exist a direction $d \in \mathbb{R}^{2n_B}$ that is a descent direction for both objectives. If $\gamma^*\in \R^{2n_B}$ satisfies this condition we call it a \textit{Pareto critical shape}. 

In this work, we aim at the efficient computation of Pareto critical shapes that, ideally,  approximate the Pareto front. 
Since derivative information can be obtained for both objective functions, we select solution methods that efficiently utilize this information and that can be adopted such that a meaningful representation of a Pareto critical front is obtained. As two fundamental approaches in this category, a parametrized  weighted sum method and a biobjective descent algorithm are chosen and explained in Sections~\ref{subsec:weightedsum} and \ref{subsec:descent}, respectively. 
Their performance  in the context of 2D shape optimization problems is compared in Section~\ref{sec:CaseStudies}.

\subsection{Weighted Sum Method}\label{subsec:weightedsum}

Maybe the easiest way to compute a representation of the Pareto front is to iteratively solve weighted sum scalarizations \eqref{eq:shapeweightedsum} with varying weights. The weighted sum scalarization of problem \eqref{discretizedMOP} can be restated as
\begin{equation} \label{eq:weightedsum}
\begin{split}
\min_{\gamma\in\Gamma}\;\; & f_w(\gamma) := \omega f_1(\gamma) + (1-\omega) f_2(\gamma)\\
\text{s.t. } & u(X(\gamma)) \text{ solves the discretized state equation } (\ref{stateequation}),
\end{split}
\end{equation}
where $\omega\in(0,1)$ is the \emph{weight} specifying the relative importance of $f_1$ and $f_2$, respectively. 
Recall that every solution of the weighted sum scalarization \eqref{eq:weightedsum} is Pareto optimal for \eqref{ceramicMOP} \cite{ehrgott}. A disadvantage of the weighted sum method is, however, that only solutions that map to the convex hull $\operatorname{conv}(Z)$ of the image set $Z=f(\Gamma)$ in the objective space can be found, and thus relevant compromise solutions in nonconvex areas of the nondominated front may be missed. Moreover, \cite{dasdennis} showed at simple biobjective test instances that evenly distributed weights do in general not lead to well distributed outcome vectors in the objective space.  
This is particularly problematic if the considered objective function values are of largely different magnitude, which is the case here. In order to obtain solutions that are consistent with the preferences expressed by $\omega$, we thus normalize the objective functions by using appropriate scaling factors $c_1,c_2>0$, and replace $f_1$ and $f_2$ in \eqref{eq:weightedsum} by $c_1f_1$ and $c_2f_2$, respectively. 

Despite the difficulties mentioned above, the weighted sum method is usually well-suited to efficiently compute at least a rough approximation of the Pareto front. For this purpose, problem \eqref{eq:weightedsum} is solved iteratively for varying weights (in our case, we choose $\omega\in\{0.2,0.25,0.3,\dots,0.9\}$ since numerical experiments showed that this yields meaningful trade-offs). 
Each single objective optimization problem \eqref{eq:weightedsum} is then individually solved 
using a classical gradient descent algorithm with stepsizes determined according to the Armijo rule, see, for example, \cite{baza:nonl:2006}. 

\begin{algorithm}[htb]
  \caption{Parametric weighted sum algorithm using gradient descent\label{alg:weightedsum}}
  \KwData{Choose $\beta\in(0,1)$, $\gamma^{(1)}\in \Gamma$, weights $\omega_1,\dots,\omega_J\in(0,1)$, 
  and  $\varepsilon > 0$. }
  \KwResult{Set of approximations of Pareto critical solutions $\tilde{\gamma}_1,\dots,\tilde{\gamma}_J$. }
  \For{$j=1$ \KwTo $J$}{
  Set $\omega=\omega_j$, set $k:=1$, and set $d^{(0)}:=-\nabla f_{\omega}(\gamma^{(1)})$ and $t_0:=1$;
  
  \While{
     $\| t_{k-1} \, d^{(k-1)}\| > \varepsilon$ 
     }{
    Compute a search direction $d^{(k)}=-\nabla f_{\omega}(\gamma^{(k)})$ ; 
    
    Compute a step length $t_k\in(0,1]$ as 
        \begin{align*}
        \max \Bigl\{t\!=\!\frac{1}{2^\ell}\; :\; \ell\in \mathbb{N}_0,\, f_{\omega}(\gamma^{(k)}\!+\!td^{(k)})\leq f_{\omega}(\gamma^{(k)})\!+\!\beta \,t \nabla f_{\omega}(\gamma^{(k)})^\top d^{(k)} \Bigr\};
        \end{align*}
        
    $\gamma^{(k+1)}:=\gamma^{(k)}+t_k\, d^{(k)}$ and $k:=k+1$\;
  }
  
  $\tilde{\gamma}_j:=\gamma^{(k)}$}
\end{algorithm}

Under appropriate assumptions, the gradient descent algorithm in the inner loop of Algorithm~\ref{alg:weightedsum} converges to a critical point of \eqref{eq:weightedsum}, see, e.g., \cite{baza:nonl:2006}. 
In our implementation, the inner loop is also terminated when a prespecified maximum number of iterations is reached. However, in this case there is no guarantee that the final iterate is close to a Pareto critical solution.

Note that a critical point of the weighted sum scalarization \eqref{eq:weightedsum} is necessarily Pareto critical for the biobjective shape optimization problem \eqref{discretizedMOP}, while the converse is not true in general. This has some correspondence to the fact that global optimal solutions of a weighted sum scalarization \eqref{eq:weightedsum} are always Pareto optimal, while nonconvex problems may have Pareto optimal solutions that are not optimal for any weighted sum scalarization \eqref{eq:weightedsum}, see, e.g., \cite{ehrgott}.

Note also that the search direction $d^{(k)}=-\nabla f_{\omega}(\gamma^{(k)})$ does not necessarily satisfy $\nabla f_j(\gamma^{(k)})^\top d^{(k)}<0$, $j=1,2$, in all iterations. In other words, one objective function may deteriorate during the optimization process if only the other objective function compensates for this. 

\subsection{Biobjective Descent Algorithm}\label{subsec:descent}

Different from the weighted sum method described above, biobjective descent algorithms -- as a natural generalization of single-objective gradient descent algorithms -- are potentially capable of finding every Pareto optimal solution, if only the starting solution is chosen appropriately. While this is a rather theoretical advantage, biobjective descent algorithms are indeed highly efficient in finding (or approximating) one Pareto critical solution 
without the necessity to specify preferences. However, if a representation of the complete Pareto front is sought, they need to be combined with other search strategies. 

We adopt the multiobjective descent algorithm proposed in \cite{Fliege2000} (see also \cite{flie:comp:2018}) 
for the biobjective optimization problem \eqref{ceramicMOP}. Similar approaches have been suggested in \cite{desideri:MGDA,desi:mult:2012,giac:comp:2014}. 

 Biobjective descent algorithms iteratively improve both objective functions simultaneously. This is based on the observation that,
 if a solution $\gamma\in\R^{2n_B}$ is not Pareto critical according to \eqref{Pareto_crit}, then there exists a direction $d \in \mathbb{R}^{2n_B}$ which is a descent direction for both objectives. 
Thus, if in an iterative solution method the current iterate $\gamma^{(k)}\in\R^{2n_B}$ is not Pareto critical, a \emph{direction of steepest biobjective descent} $d^{(k)} \in \mathbb{R}^{2n_B}$ can be defined according to \cite{Fliege2000} as  a direction solving the auxiliary optimization problem
\begin{equation}
\begin{aligned}
& \min_{\rho\in\R, d\in\R^{2n_B}}
& & \rho + \frac{1}{2}\|d\|^2 \\
& \text{s.t.}
& &  \nabla f_j(\gamma^{(k)})^\top\,d \leq \rho ,\ j=1,2.
\end{aligned}
\label{search_dir_quad}
\end{equation}
Problem~\eqref{search_dir_quad} is a convex quadratic optimization problem with linear inequality constraints.  Note that the term $\frac{1}{2}\|d\|^2$ in the objective function  ensures that the problem is bounded, and that the solution $\rho=0$, $d=0$ is always feasible.  
Note also that the optimal value $\rho^*$ is negative if and only if $d^*\neq 0$, i.e., if a direction of steepest biobjective descent exists. 

When a direction of steepest biobjective descent $d^{(k)}\neq 0$ is found, then we move from $\gamma^{(k)}$ into the direction $d^{(k)}$ to a new point $\gamma^{(k+1)}:=\gamma^{(k)}+t_kd^{(k)}$.
The step length $t_k>0$ is computed using an Armijo-like rule. Towards this end, let $\beta \in (0,1)$ be a prespecified constant. Then a step length $t$ is accepted if it guarantees a sufficient biobjective descent in the sense that
\begin{align}
f_j(\gamma^{(k)}+t\,d^{(k)})\leq f_j(\gamma^{(k)})+\beta\, t\, \nabla f_j(\gamma^{(k)})^\top\,d^{(k)}, \; j=1,2.
\label{armijo_step}
\end{align}
In order to compute an acceptable step length $t$, we iteratively test the values $(\frac{1}{2})^{\ell}$, $\ell=0,1,2,\dots$ until condition \eqref{armijo_step} is satisfied. A proof for the finiteness of this procedure is given in \cite{Fliege2000}. 
The overall method is summarized in Algorithm~\ref{alg:fliege}.

\begin{algorithm}[htb]
  \caption{Biobjective descent algorithm according to \cite{Fliege2000}\label{alg:fliege}}
  \KwData{Choose $\beta\in(0,1)$, $\gamma^{(1)}\in \Gamma$ and  $\varepsilon > 0$, set $k:=1.$   }
  \KwResult{Approximation of a Pareto critical solution $\tilde{\gamma}:=\gamma^{(k)}$. }
  
  Compute $d^{(0)}:=d^{(1)}$ as a solution of \eqref{search_dir_quad} and set $t_0:=1$;

  \While{
     $\| t_{k-1}\, d^{(k-1)}\| > \varepsilon$ 
     }{
    Compute $d^{(k)}$ as a solution of \eqref{search_dir_quad}; 
    
    Compute a step length $t_k\in(0,1]$ as 
        \begin{align*}
        \max \Bigl\{t\!=\!\frac{1}{2^\ell}\; :\; \ell\!\in \mathbb{N}_0,\, f_j(\gamma^{(k)}\!+\!td^{(k)})\leq f_j(\gamma^{(k)})\!+\!\beta t \nabla f_j(\gamma^{(k)})^\top d^{(k)}, \; j\!=\! 1,2 \Bigr\};
        \end{align*}
        
    $\gamma^{(k+1)}:=\gamma^{(k)}+t_k\, d^{(k)}$ and $k:=k+1$\;
  }
\end{algorithm}

If $f_1$ and $f_2$ are continuously differentiable and $\varepsilon=0$, then Algorithm~\ref{alg:fliege} converges to a Pareto critical solution \cite{Fliege2000}. 
A natural stopping condition for practical implementations, motivated by \eqref{Pareto_crit}, is that
$\| t_k d^{(k)}\| \leq \varepsilon$, with $\varepsilon > 0$ a  prespecified small constant.

In practice, we also terminate the algorithm when  
a prespecified maximum number of iterations is reached. In this case, the final solution has to be used with caution since the optimization procedure has generally not yet converged. 

The choice of the search direction using problem \eqref{search_dir_quad} together with condition \eqref{armijo_step} implies that the iterates of Algorithm~\ref{alg:fliege} satisfy
$f(\gamma^{(k+1)})<f(\gamma^{(k)})$ for all $k=1,2,\dots$. In other words, 
the objective vector $f(\gamma^{(k+1)})$ in iteration $k+1$ is bounded above by the objective vector $f(\gamma^{(k)})$ of the previous iteration $k$, i.e., $f(\gamma^{(k+1)}) \in f(\gamma^{(k)}) -\mathbb{R}^2_{>}$. 

Several alternative Pareto critical solutions (and hence trade-off information between them) can be obtained, for example, by varying the starting solution. We follow a different approach in our implementation that is somewhat similar to the weighted sum method, and that is based on the observation that the optimal solution of problem \eqref{search_dir_quad} (i.e., the direction of steepest biobjective descent) depends on the scaling of the objective functions $f_1$ and $f_2$. Thus, Algorithm~\ref{alg:fliege} is executed repeatedly, using different scalings of the objective functions. 
In our implementation, we use a scaling parameter $s:=\bar{\omega}r^{\max}>0$ and replace $f_2$ by $sf_2$ in the optimization process, where the parameter $r^{\max}>0$ is chosen as the largest ratio between partial derivatives of $f_1$ and $f_2$, evaluated at the starting solution $\gamma^{(1)}$. 
Note that the latter aims at the constraints in problem \eqref{search_dir_quad} in the sense that they should be comparable, i.e., both objective functions should equally contribute to active constraints and thus influence the choice of the search direction. By varying the parameter $\bar{\omega}\in \{0.5,0.6,\dots,2\}$, we can compute different solutions starting from the same initial shape. Note that the volume of the solutions can be expected to increase with larger values of $\bar{\omega}$. 
%
%
%

Note also that the resulting parametric version of Algorithm~\ref{alg:fliege} is fundamentally different from the weighted sum method in Algorithm~\ref{alg:weightedsum} in the way the search directions are chosen and in the way the iterates converge to a Pareto critical solution. 

\subsection{Scalar Products and Gradients in Shape Optimization}
\label{subsec:Gradients}
The performance of Algorithms~\ref{alg:weightedsum} and \ref{alg:fliege} depends largely on the choice of the search direction, which is computed based on the discretized gradients $\nabla f_j(\gamma)$, $j=1,2$. 
Michor and Mumford \cite{michor} showed that (continuous) shape gradients calculated with respect to the ordinary $L^2$-scalar product lead to an ill defined notion of the distance of two shapes, as the infimum over all deformation path lengths is zero. They suggest a modified scalar product given by 
\begin{equation}\label{eq:curvature}
    \langle h,k\rangle_\xi=\int_{\partial\Omega}\langle h,k\rangle_{\mathbb{R}^2} \, (1+\xi \kappa^2) \, \text{d}A
\end{equation}
and show that this indeed leads to a well defined Riemannian metric on the shape space. 
Here, $h,k$ are two vector fields in normal direction to the boundary of $\partial\Omega$, $\text{d}A$ is the induced surface measure, $\kappa$ is the scalar curvature of the surface, and $\xi>0$ is a regularization parameter. 
In practice, this corresponds to a transformation of function values on $\partial\Omega$ that, given some function $g:\partial\Omega\rightarrow\R^2$, can be described by 
$g_\xi(x)=\frac{g(x)}{1+\xi\kappa^2(x) }$ for $x\in\partial\Omega$. 


We adopt a discretized version of this concept in the numerical implementation of shape gradients for both objectives $f_j$, $j=1,2$. More precisely, a discretized scalar curvature $\kappa$ is computed at grid points on the boundary $\partial\Omega$, which is represented by a polygonal approximation induced by the shape parameters $(\varrho^{\text{ml}}, \varrho^{\text{th}})\in \R^{2 n_x}$, $\varrho^{\text{th}}\in \R_>^{n_x}$. Since the upper and lower boundary of the shape $\Omega$ may have a different curvature at the same $x$-coordinate value $x_i$ ($i\in\{1,\dots,n_x\}$), we have to compute the curvature for upper and lower boundary points separately. For the upper boundary, this is realized by comparing the normals $n^{\text{u}}_i$ and $n^{\text{u}}_{i+1}$ on two consecutive facets of length $l^{\text{u}}_i$ and $l^{\text{u}}_{i+1}$, respectively. Similarly, for the lower boundary we use $n^{\text{l}}_i$, $n^{\text{l}}_{i+1}$ and $l^{\text{l}}_i$, $l^{\text{l}}_{i+1}$, and obtain
\begin{equation}
\begin{array}{ll}
\begin{array}{r@{\extracolsep{0.5ex}}c@{\extracolsep{0.5ex}}r@{\extracolsep{0.5ex}}c@{\extracolsep{0.5ex}}l}
\kappa^{\text{u}}_i&:=&\kappa^{\text{u}}(x_i) &=& \displaystyle\frac{2\| n^{\text{u}}_i -n^{\text{u}}_{i+1}  \|_2}{l^{\text{u}}_i+l^{\text{u}}_{i+1}},\\[0.5cm]
    \kappa^{\text{l}}_i&:=&\kappa^{\text{l}}(x_i)&=&\displaystyle\frac{2\| n^{\text{l}}_i -n^{\text{l}}_{i+1}  \|_2}{l^{\text{l}}_i+l^{\text{l}}_{i+1}},
\end{array} & \qquad i=1,\dots,n_x-1 .
\end{array}
    \label{kappa}
\end{equation}
The upper and lower boundaries of the shape $\Omega$ are reconstructed from the meanline and thickness representation using the linear transformation $\varrho_i^{\text{u}}=\varrho_i^{\text{ml}}+\frac{1}{2}\varrho_i^{\text{th}}$ and $\varrho_i^{\text{l}}=\varrho_i^{\text{ml}}-\frac{1}{2}\varrho_i^{\text{th}}$, $i=1,\dots, n_x$. In other words, $(\varrho^{\text{u}}, \varrho^{\text{l}})\in \R^{2 n_x}$ is obtained from $(\varrho^{\text{ml}}, \varrho^{\text{th}})\in \R^{2 n_x}$, $\varrho^{\text{th}}\in\R_>^{n_x}$, as  $(\varrho^{\text{u}}, \varrho^{\text{l}})=M\, (\varrho^{\text{ml}}, \varrho^{\text{th}})$, using an appropriate transformation matrix $M \in \mathbb{R}^{{2n_x}\times {2n_x}}$. 
This leads to a discretized representation of the respective boundaries by points $(x_i,\varrho_i^{\text{u}})$ (upper boundary) and
$(x_i,\varrho_i^{\text{l}})$ (lower boundary), from which the $\kappa$ values can be computed according to \eqref{kappa}. 

Now \eqref{eq:curvature} can be applied to the gradients 
of $f_j$ w.r.t.\ $(\varrho^{\text{u}}, \varrho^{\text{l}})$, $j=1,2$, 
by multiplying the respective partial derivatives by 
\begin{equation*}
    d_{\xi,i}^{\text{u}}:=\frac{1}{1+\xi\,(\kappa^{\text{u}}_i)^2} \quad\text{and}\quad d_{\xi,i}^{\text{l}}:=\frac{1}{1+\xi\,(\kappa^{\text{l}}_i)^2}, \quad i=1,\dots, n_x.
\end{equation*}
Since we actually need the gradients of $f_j$ w.r.t.\ $\varrho=(\varrho^{\text{ml}}, \varrho^{\text{th}})$, $j=1,2$, we additionally have to consider the linear tranformation $M$. 
Let $D_\xi=(d_{\xi,ij})_{2n_x\times 2n_x}\in \mathbb{R}^{{2n_x}\times {2n_x}}$ 
be a diagonal matrix with diagonal elements given by
\begin{equation*}
    d_{\xi,ii}:=d_{\xi,i}^{\text{u}},\;\; i=1,\dots, n_x \quad\text{and}\quad d_{\xi,ii}:=d_{\xi,i-n_x}^{\text{l}},\;\; i=n_x+1,\dots, 2n_x,
\end{equation*}
and set $\bar{D}_\xi:=M^{-1}\,D_\xi\,M$. Then we obtain the curvature adapted B-spline gradients as
\begin{equation}
\Bigl(\frac{\partial f_j}{\partial\gamma}\Bigr)_{\xi}
    =\bar{D}_\xi\left(\frac{\partial f_j}{\partial X} \, \frac{\partial X}{\partial\varrho}\right)\frac{\partial\varrho}{\partial\gamma},\qquad j=1,2.
\end{equation}
Note that for $\xi=0$ the matrix $\bar{D}_0$ is the identity matrix, and hence the $L^2$-gradient of $f_j$ w.r.t.\ $\gamma$, $j=1,2$, is recovered in this case,  c.f.\ \eqref{eq:L2gradient}.

\subsection{Control of Step Sizes}
\label{subsec:Steps}
Large mesh deformations may cause numerical difficulties and thus have to be avoided. We thus limit the step size during the optimization procedure. Recall that the representation of feasible shapes, using meanline and thickness values $(\varrho^{\text{ml}}_i,\varrho^{\text{th}}_i)$ at fixed $x_i$ coordinates, $i=1,\dots,n_x$, implies that grid points can only move vertically. A natural choice for a maximum admissible step in one iteration of the optimization process is thus determined by the thickness of the shape, divided by the number $n_y$ of gridpoints in $y$-direction. Since in our case studies the shapes are fixed at the left boundary (i.e., at $x=x_1$) and hence their thickness is constant at $x_1$, we set 
\[\delta^{\max} := 0.8 \, \frac{\varrho_{1}^{\text{th},(1)}}{n_y}\]
i.e., to $80\%$ of the vertical distance between grid points on the left boundary of the initial shape. 
For a given search direction $d^{(k)}=(d^{\text{ml},(k)}, d^{\text{th},(k)}) \in \R^{2 n_B}$ in iteration $k$ of the optimization algorithms, we check whether  $\max_{i=1,\dots,2n_B} |d^{(k)}_i |\leq \delta^{\max}$. 
Otherwise,  $d^{(k)}$ is scaled by a factor ${\delta^{\max}}/{\max_{i=1,\dots,2n_B} |d^{(k)}_i|}$. 
Then the step length $t\leq 1$ is computed according to the Armijo rule as indicated in Algorithms~\ref{alg:weightedsum} and \ref{alg:fliege}.

While $\delta^{\max}$ is derived from the mesh $X^{(1)}$, it still is a meaningful upper bound for a step $d^{(k)}$ in the B-spline representation. Indeed, 
if $\{\vartheta_j,\, j=1, \dots, n_B  \}$ is a B-spline basis and $\gamma^{(k)}=(\gamma^{\text{ml},(k)}, \gamma^{\text{th},(k)}) \in \Gamma$ is the current iterate, then the B-spline basis properties $\sum_{j=1}^{n_B}\vartheta_j(x)=1$ and $\vartheta_j(x)\geq0$, $j=1,\dots,n_B$ (see, e.g., \cite{nurbs}) imply that, for all $i=1,\dots,n_x$,
\begin{eqnarray*}
\lefteqn{    \left|  \varrho_i^{\text{ml},(k+1)} - \varrho_i^{\text{ml},(k)}  \right|
    = \Bigl|\sum_{j=1}^{n_B}(\gamma_j^{\text{ml},(k)}\!\!+\!d_j^{\text{ml},(k)})\vartheta_j(x_i)\!-\!\!\sum_{j=1}^{n_B}\gamma_j^{\text{ml},(k)}\vartheta_j(x_i)\Bigr| }\\
    &&\leq \sum_{j=1}^{n_B} |d_j^{\text{ml},(k)}|\,|\vartheta_j(x_i)| 
    \leq \max_{j=1,\dots,n_B}|d_j^{\text{ml},(k)}|\,  \sum_{j=1}^{n_B} |\vartheta_j(x_i)|
    = \max_{j=1,\dots,n_B}|d_j^{\text{ml},(k)}| .
\end{eqnarray*}
An analogous bound holds for the corresponding thickness parameters. 
Note that the above inequalities do in general not guarantee that \emph{all} grid points of the corresponding mesh $X^{(k)}$ move by at most $80\%$, since this also depends on 
the current shape and 
the mutual movement of meanline and thickness values. In some situations it may thus be necessary to adapt this bound to a smaller value. However, this never occured in our numerical tests. 

\section{Case Studies}
\label{sec:CaseStudies}
We consider 2D ceramic shapes made out of beryllium oxide (BeO) under tensile load.
Therefore, we set Young's modulus to $\texttt{E}=320\,\text{GPa}$ (see, e.g., \cite{munz}), 
Poisson's ratio to $\nu=0.25$, 
and the ultimate tensile strength to $140\,\text{MPa}$, according to \cite{crcmaterials}. 
Weibull's modulus is set to $m=5$, which is on the lower bound of industrial ceramics having $m$ between 5 and 30 as depending on the production process \cite{keramikverband}. 
All considered shapes have a fixed length of  $1.0\,\text{m}$ and a fixed height of $0.2\,\text{m}$ on the left and right boundaries. The shapes are fixed on the left boundary, where Dirichlet boundary conditions hold ($\partial\Omega_{D}$), and on the right boundary, where surface forces may act on and Neumann boundary conditions 
hold  ($\partial\Omega_{N_{\text{fixed}}}$). The upper and lower boundaries are assumed to be force free ($\partial\Omega_{N_{\text{free}}}$). They can be modified within the optimization process. 
We set $\tilde{f} = 0$ neglecting the gravity forces and $\tilde{g} = 10^{7}\,\text{Pa}$,   representing tensile load. 
Note that, in order to be  consistent with 3D models, we define the force density w.r.t.\  $\text{Pa}=\text{N}/\text{m}^2$ (and not w.r.t.\  $\text{N}/\text{m}$). This is motivated by assuming a  constant width of the 2D component of $1$ unit (i.e., $1\text{m}$). 
Then plane stresses and plane strains are obtained by neglecting Poisson effects in the third dimension. 
%

The shapes are discretized by a $41\times 7$ grid (i.e., $n_x=41$ and $n_y=7$) using tetrahedrons as detailed in Section~\ref{subsec:Geometry}. The B-spline representation is based on $n_B=5$ basis functions. 
Moreover, the curvature regularization paramater is set to $\xi= 10^{-4}$, see Section~\ref{subsec:Gradients}.

During the optimization process, we monitor the Euclidean norm of the update of the design variables in every iteration and stop when it is lower than  $10^{-4}$. 
The implementation is realized in R version 3.4.2 and uses the adjoint finite element code of \cite{PaperHahn} as a subroutine.

\subsection{A Straight Joint} 
\label{subsec:Beam}

In the first test case, a straight joint is sought that is fixed at the left side, while the tensile load acts on the right side. This is a particularly simple situation where the straight rod connecting from the left to the right can be expected to be optimal, with varying thickness depending on the trade-off between the intensity measure ($f_1$) and the volume ($f_2$). The optimization algorithms are challenged by providing a bended beam as a starting shape, which is clearly far from being optimal. 

The starting shape is shown in Figure~\ref{fig:TC1_Start},  together with the $41\times 7$ tetrahedral discretization $X$. Its objective values are $f_1(X^{(1)})=0.769624$ (intensity measure) and $f_2(X^{(1)})=0.2$ (volume), respectively. The relatively high value for the intensity measure $f_1$ can be explained by the relatively high stresses that are illustrated in Figure~\ref{fig:TC1_Start_sigma}. Figure~\ref{fig:TC1_Start_bspline_sigma} shows that the B-spline representation based on only five basis functions leads to a rather inaccurate representation, particularly at the left and right boundary. This could be improved by fixing the slopes at the left and right boundary, however, at the price of a significantly reduced design space. Indeed, a majority of the Pareto critical shapes computed during our numerical tests do not have zero slopes at the left and right boundary, particularly in the case of the S-shaped joint considered in Section~\ref{subsec:Joint} below.
Note that the smoothing induced by the B-spline representation in this case already leads to dominating objective values of $f_1(\gamma^{(1)})=0.453867$ and $f_2(\gamma^{(1)})=0.2$. 

\begin{figure}[ht]
	\begin{center}
		\subfloat[Starting shape: Tetrahedral mesh $X$\label{fig:TC1_Start} ]{\includegraphics[width=0.3\textwidth]{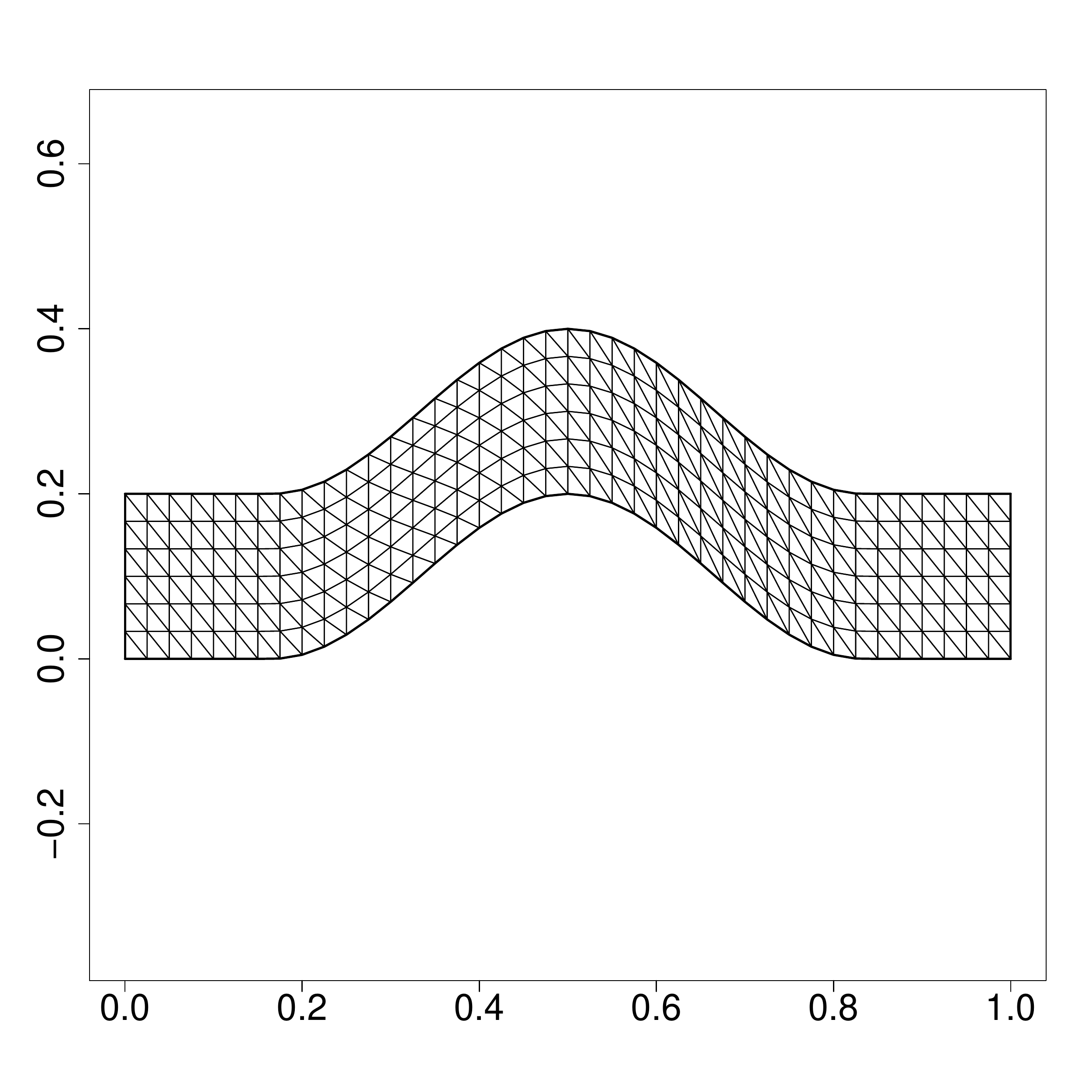}}
		\hspace{\fill}
		\subfloat[Starting shape: Objective values and stresses \label{fig:TC1_Start_sigma} ]{\includegraphics[width=0.3\textwidth]{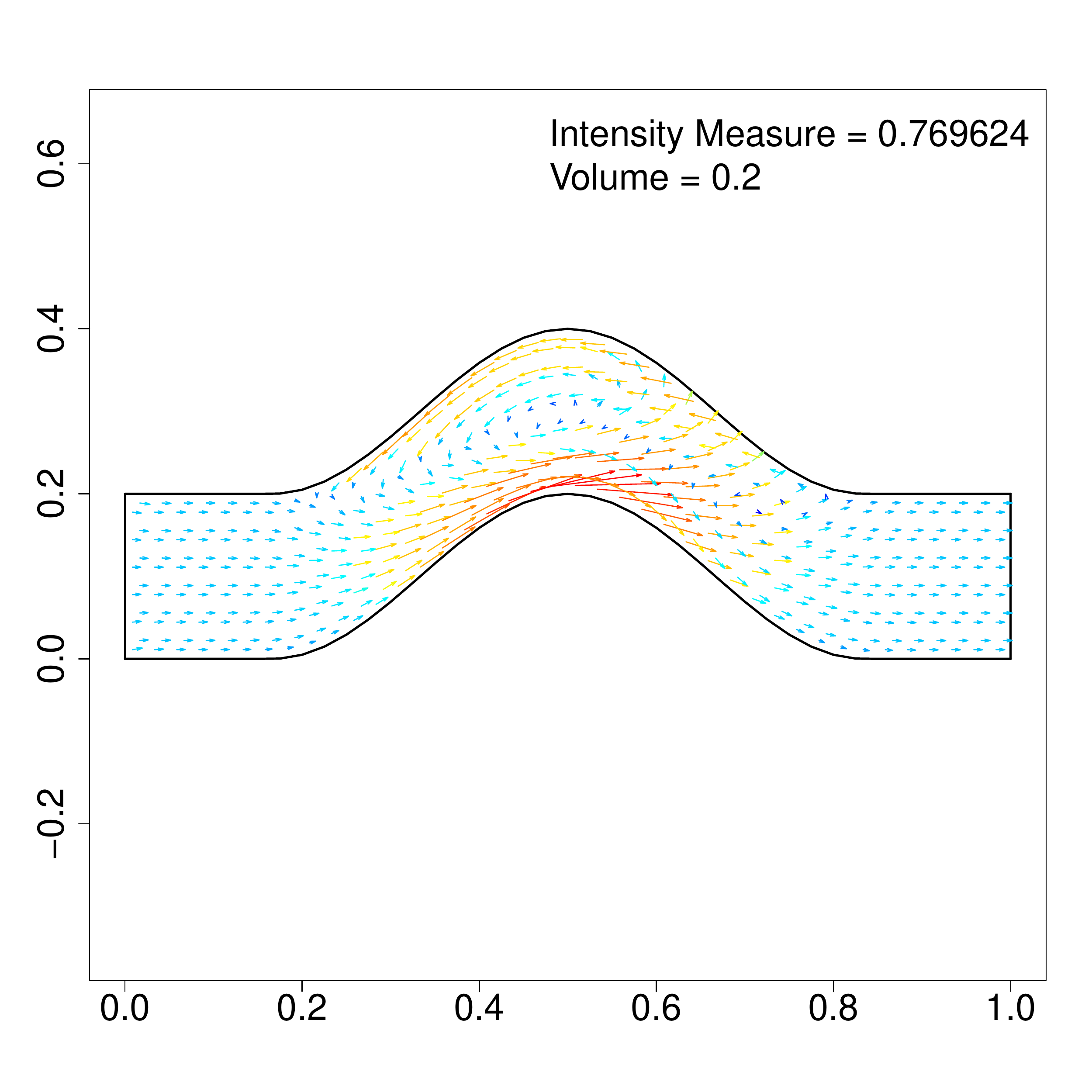}}
		\hspace{\fill}
		\subfloat[Starting shape: Approximation with B-splines \label{fig:TC1_Start_bspline_sigma} ]{\includegraphics[width=0.3\textwidth]{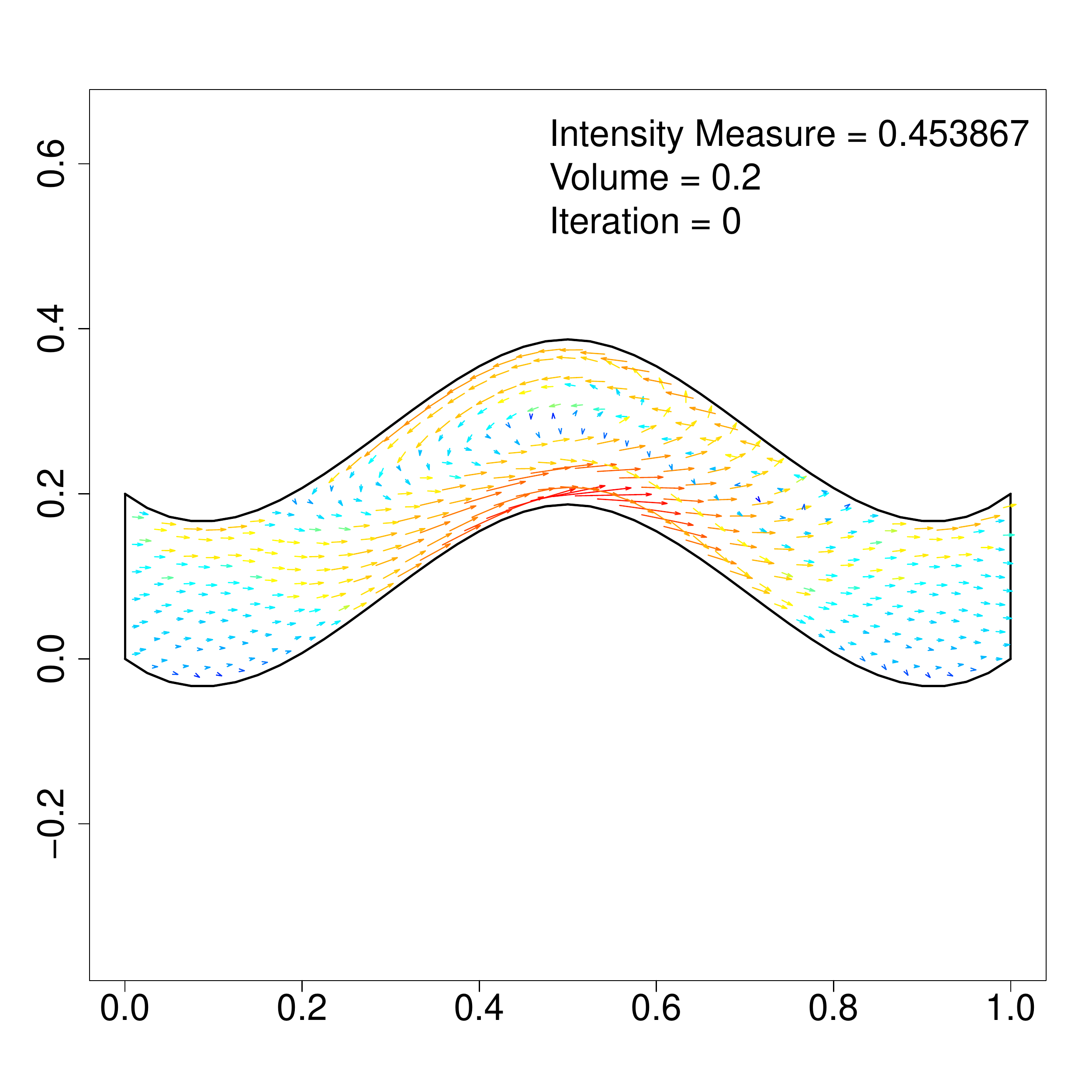}}
		\hspace{\fill}
		\subfloat[Expected result: Straight rod\label{fig:TC1_PerfectRod}
		]{\includegraphics[width=0.3\textwidth]{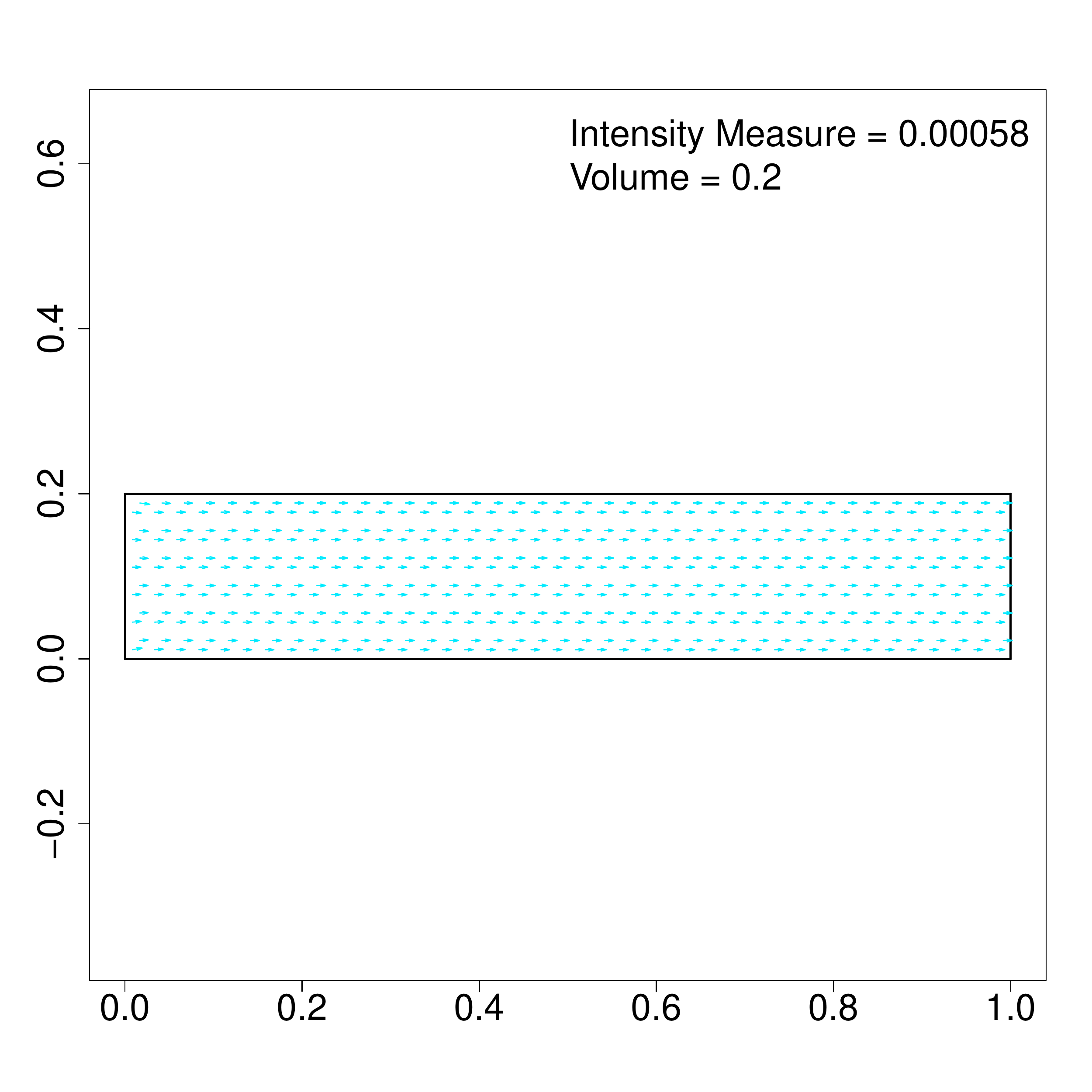}}
		\hspace{\fill}	
		\subfloat[\added{Weighted sum, $\omega = 0.8$} \label{fig:straightrodwsum8}
		]{\includegraphics[width=0.3\textwidth]{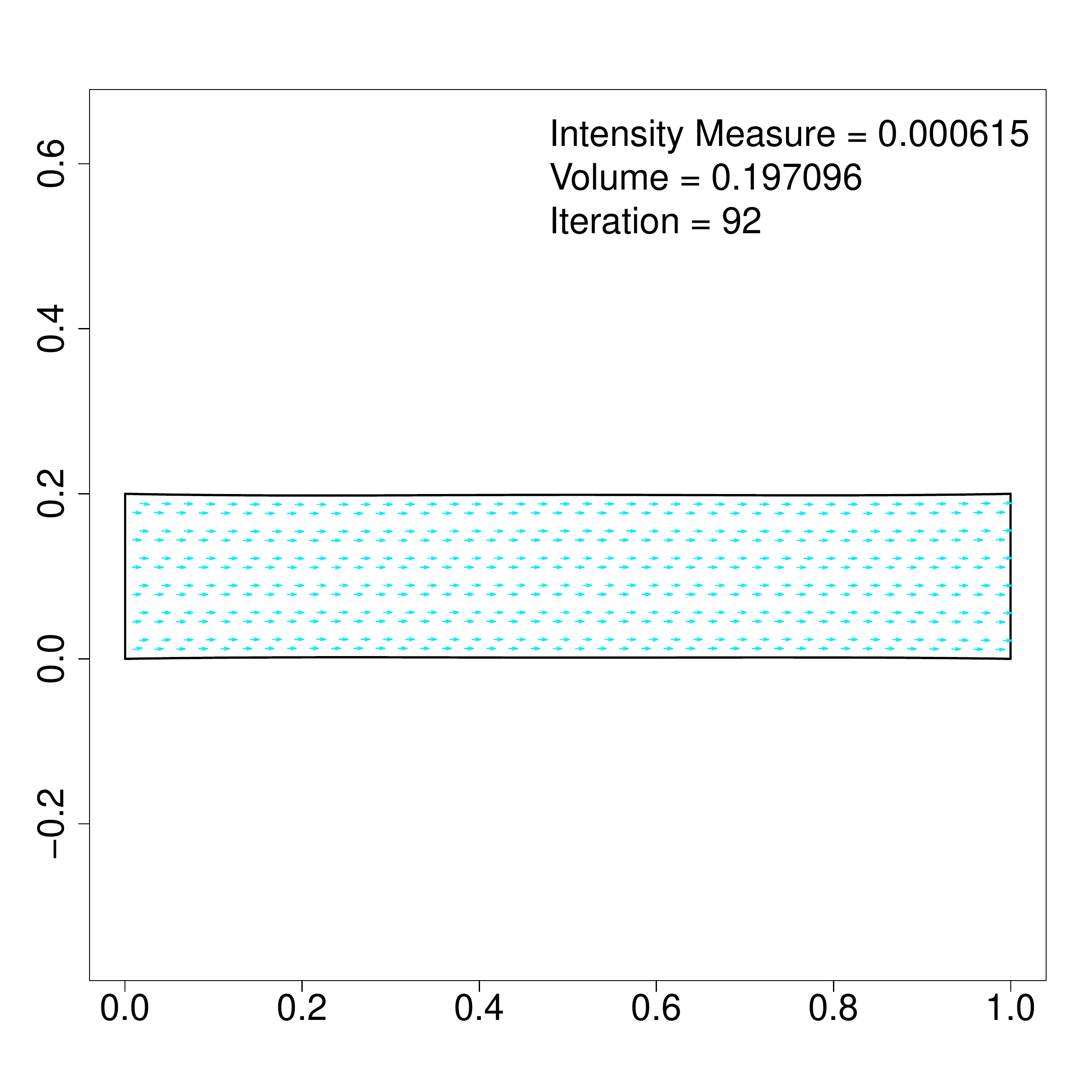}}
		\hspace{\fill}
		\subfloat[MO descent, $\bar{\omega} = 1.8$ \label{fig:straightrodmoda18}
		]{\includegraphics[width=0.3\textwidth]{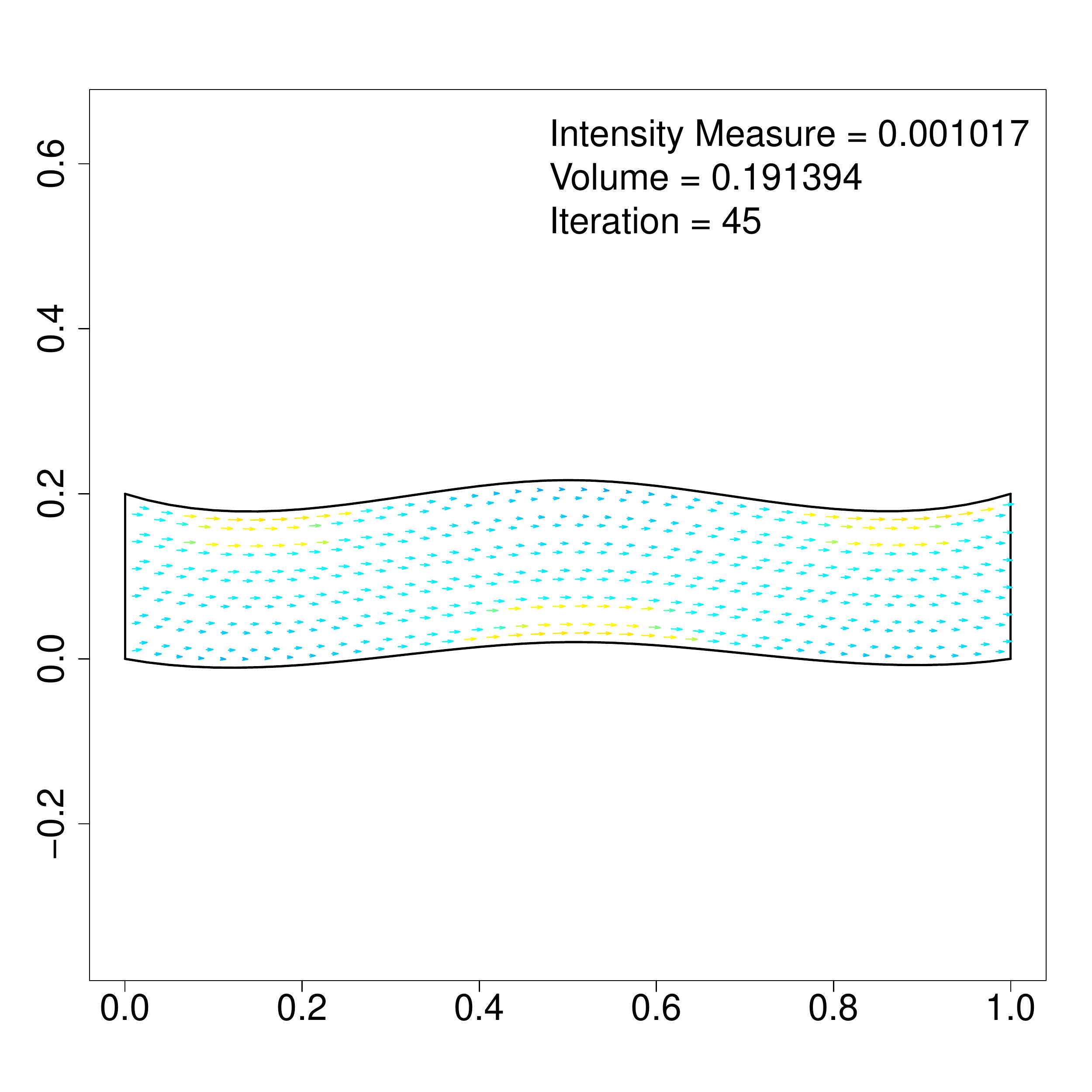}}
		\hspace{\fill}
		\subfloat[MO descent, $\bar{\omega} = 0.5$ \label{fig:straightrodmoda05}
		]{\includegraphics[width=0.3\textwidth]{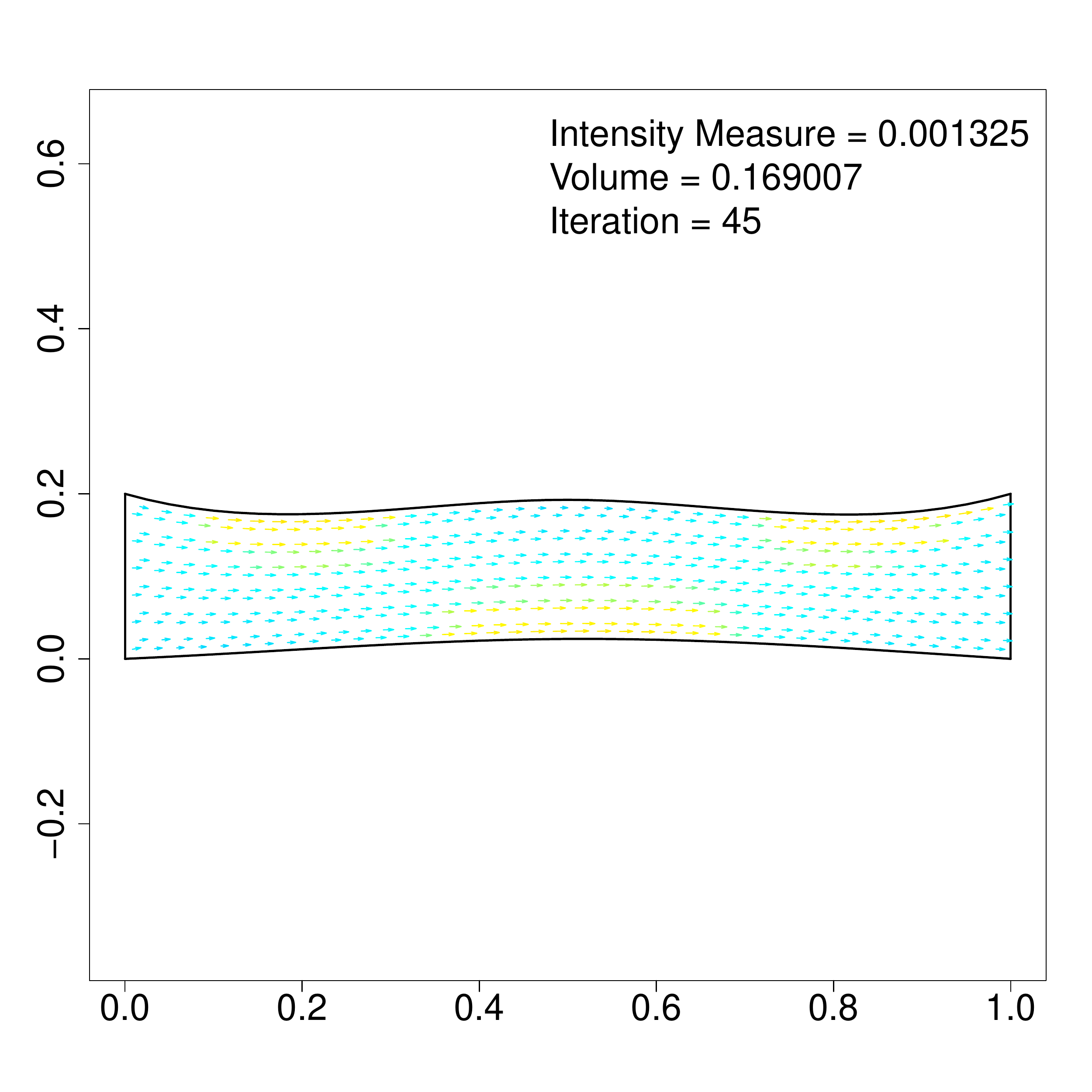}}
		\hspace{\fill}
		\subfloat[\added{Weighted sum, $\omega = 0.6$} \label{fig:straightrodwsum6}
		]{\includegraphics[width=0.3\textwidth]{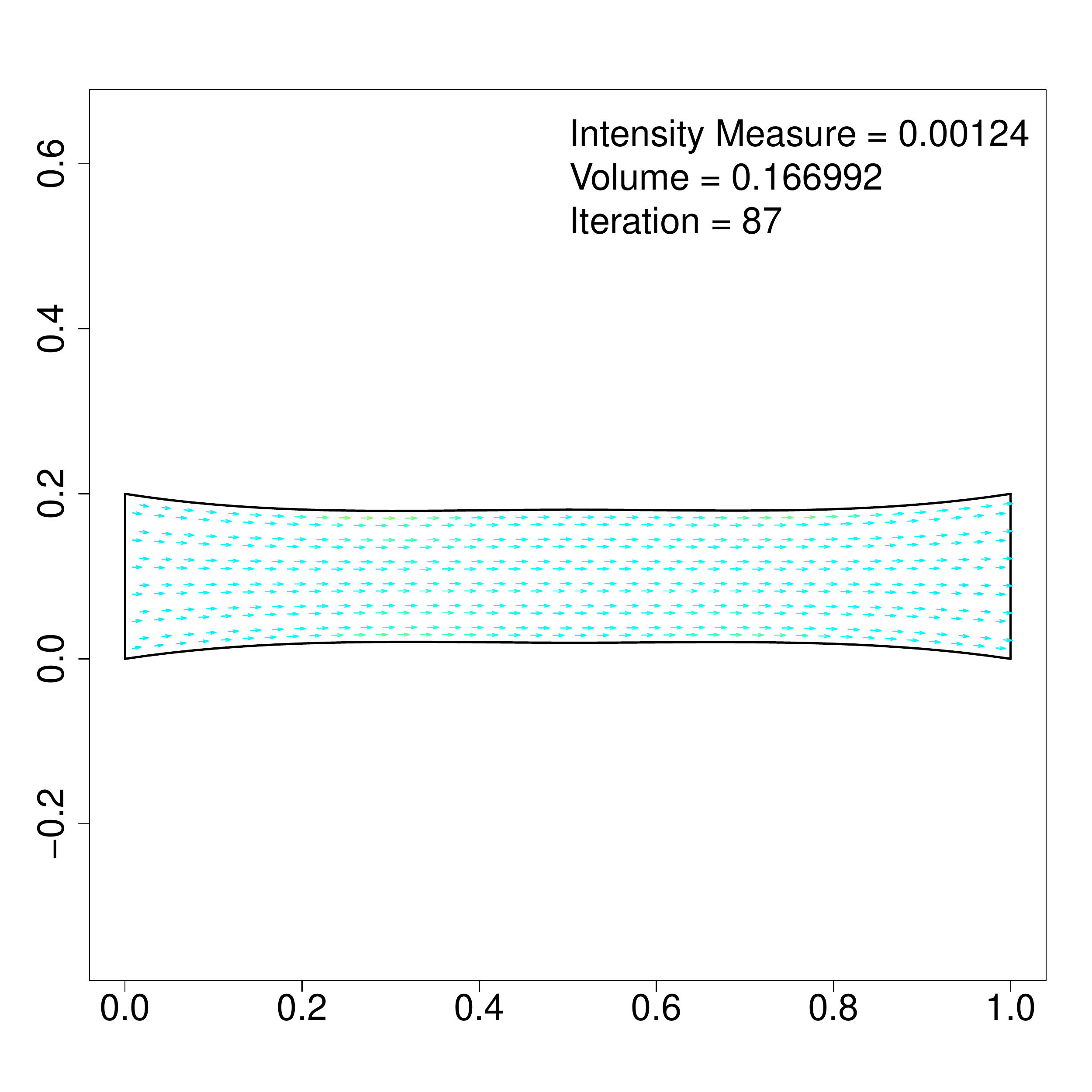}}
		\hspace{\fill}
		\subfloat[\added{Weighted sum, $\omega = 0.3$} \label{fig:straightrodwsum3}
		]{\includegraphics[width=0.3\textwidth]{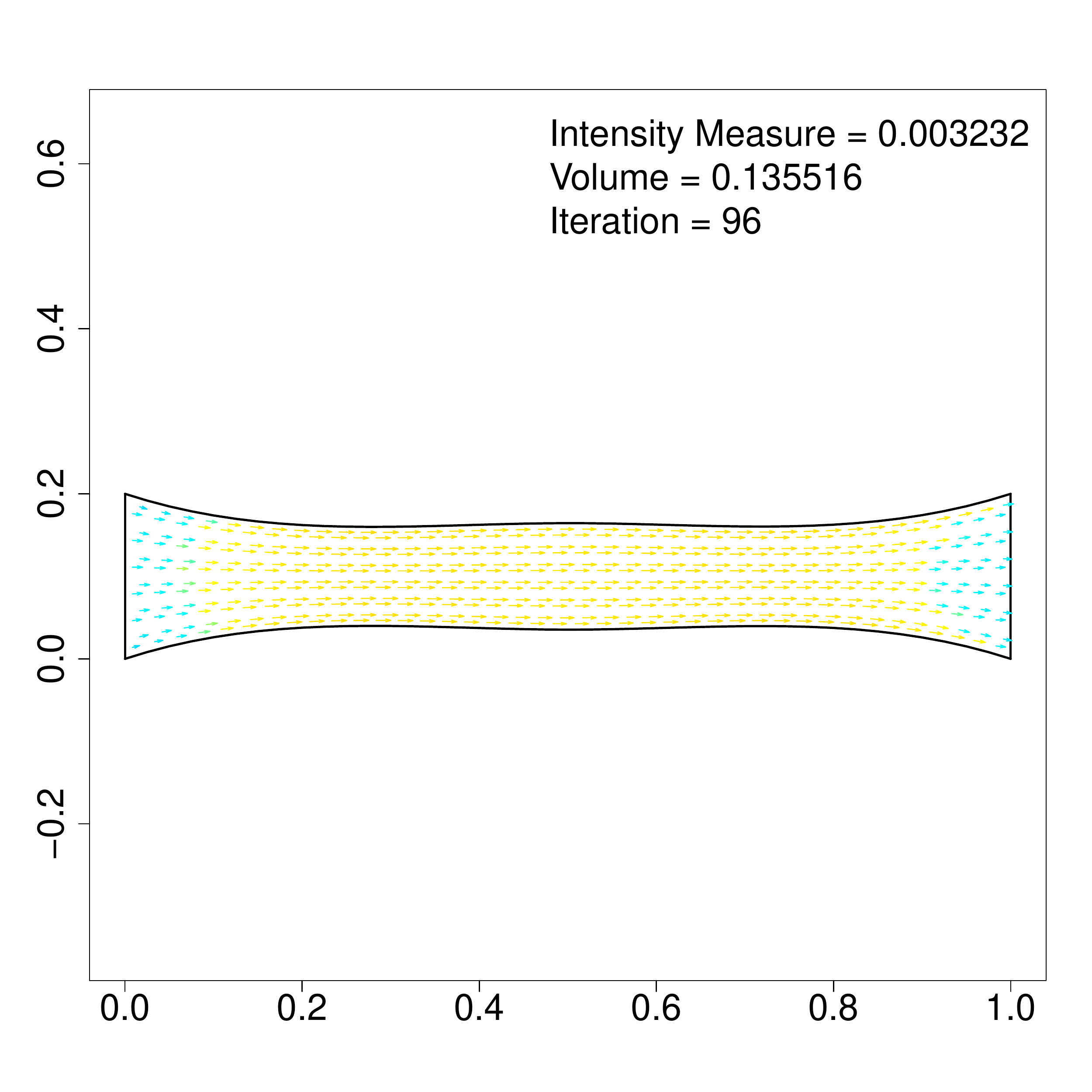}}
	\end{center}		
	\caption{Straight joint: Starting solution (row 1), straight rod solutions (row 2), and low volume solutions (row 3). \label{fig:TC1}}
\end{figure}

\paragraph{Results} 
Among all shapes with a fixed volume of $f_2(X)=0.2$, the straight rod shown in Figure~\ref{fig:TC1_PerfectRod} can be expected to have the mimimum possible intensity measure $f_1$. Indeed, the straight rod shown in Figure~\ref{fig:TC1_PerfectRod} achieves an objective value of $f_1(X)=0.00058$.  Figures~\ref{fig:straightrodwsum8} and \ref{fig:straightrodmoda18} show the results of the weighted sum method (Algorithm~\ref{alg:weightedsum}) with weight $\omega=0.8$ and of the biobjective descent algorithm (Algorithm~\ref{alg:fliege}) with scaling parameter $\bar{\omega}=1.8$. Both methods show a rather quick convergence (with the expected advantage for the biobjective descent algorithm) to solutions that are close to optimal. However, the solution of the biobjective descent algorithm seems to be a local solution with slightly higher stresses (and thus slightly higher objective value for $f_1$). 

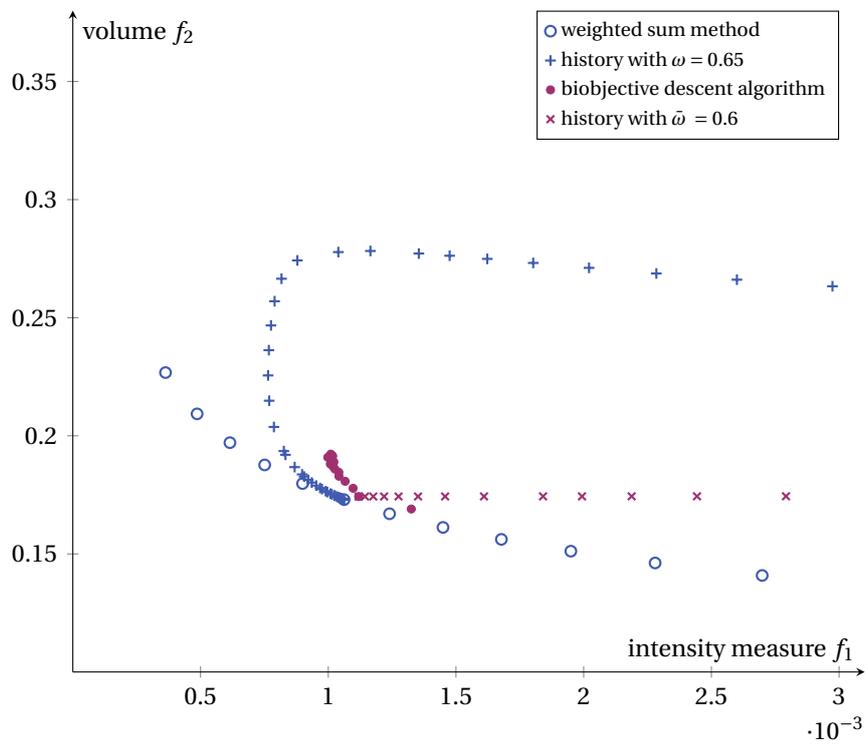
\begin{figure}
\definecolor{c1}{HTML}{3FA137}
\definecolor{c2}{HTML}{405CB0}
\definecolor{c3}{HTML}{A0306F}
\definecolor{c4}{HTML}{C7A748}
\centering
\begin{tikzpicture} 
\begin{axis}[axis lines=middle,
    width=12cm,
    xmin=0, 
    xmax=0.0031, 
    ymin=0.10,
    ymax=0.38,
	label style={font=\footnotesize}, 
	tick label style={font=\footnotesize},
    xlabel=intensity measure $f_1$, 
    ylabel={volume $f_2$},
    legend style={cells={anchor=west},font=\scriptsize, at={(axis cs:0.003,0.38)}}
    ] 
    
\addplot[only marks, color=c2,mark=o, thick]	table[x=V1,y=V2]{Data/RM01_wsm.txt};
\addplot[only marks, color=c2,mark=+, thick]	table[x=V1,y=V2]{Data/RM01_wsmhist.txt};
\addplot[only marks, color=c3,mark=*, mark size=1.5pt]	table[x=V1,y=V2]{Data/RM01_moda.txt};
\addplot[only marks, color=c3,mark=x, thick]	table[x=V1,y=V2]{Data/RM01_modahist.txt};

\legend{weighted sum method,history with $\omega = 0.65$, biobjective descent algorithm,  history with $\bar{\omega}\ = 0.6$}
\end{axis} 
\end{tikzpicture}
\caption{Iteration histories of an exemplary run \added{of the weighted sum method (Algorithm \ref{alg:weightedsum}) } and of the biobjective descent algorithm (Algorithm \ref{alg:fliege}). \label{fig:histories}}
\end{figure}

Figure~\ref{fig:histories} shows iteration histories of exemplary runs of the weighted sum method (Algorithm~\ref{alg:weightedsum})  and of the biobjective descent algorithm (Algorithm~\ref{alg:fliege}), respectively. It nicely illustrates that, in contrast to the biobjective descent algorithm, the weighted sum method 
permits iterations where one objective function deteriorates while the weighted sum objective is still decreasing. This may, in certain situations, help to overcome local Pareto critical solutions. On the other hand, the weighted sum method may get stuck in local minima as well. 
Indeed, independent of the chosen weight, the histories of the weighted sum method have a similar structure: First mainly the intensity measure (representing the PoF) is improved (since in early stages of the algorithm the gradient of $f_1$ is considerably larger than the gradient of $f_2$). Only at later stages of the algorithm, the volume is varied to a larger extent, depending on the given weight.

Note also that the final solution obtained with the biobjective descent algorithm largely depends on the starting solution, since the objective values can never deteriorate during the optimization process. Thus, when the starting solution has a volume of $f_2(X)=0.2$, then all Pareto critical shapes that can be computed with the biobjective descent algorithm  have a volume of at most $0.2$, irrespective of the scaling.

Three shapes with progressively reduced volume (and hence lower cost) are shown in Figures~\ref{fig:straightrodmoda05} to \ref{fig:straightrodwsum3}. As was to be expected, a lower cost comes at the price of a higher intensity measure (and hence higher PoF). A comparison between Figures~\ref{fig:straightrodwsum6} and \ref{fig:straightrodmoda05} suggests that also for the low volume solutions, the weighted sum solutions slightly outperform the biobjective descent solutions.

Figure~\ref{objective_space:bended_beam} summarizes the results of several optimization runs with varying weights (Algorithm~\ref{alg:weightedsum}) and varying scalings (Algorithm~\ref{alg:fliege}), respectively. The same starting solution was used in all cases, see Figure~\ref{fig:TC1_Start_bspline_sigma}. While the solution quality of the weighted sum method and of the biobjective descent algorithm is comparable, a clear advantage of the weighted sum method seams to be that it is not so much constrained by the (performance of the) starting solution. Indeed, the weighted sum solutions shown in  Figure~\ref{objective_space:bended_beam} span a large range of alternative objective values in the objective space and thus provide the decision maker with meaningful trade-off information and a variety of solution alternatives.

\begin{figure}
    \centering
\definecolor{c1}{HTML}{3FA137}
\definecolor{c2}{HTML}{405CB0}
\definecolor{c3}{HTML}{A0306F}
\definecolor{c4}{HTML}{C7A748}
\begin{tikzpicture} 
\begin{axis}[axis lines=middle,
width=12cm,
xmin=0, 
xmax=0.00399, 
ymin=0.12,
ymax=0.25,
label style={font=\footnotesize}, 
tick label style={font=\footnotesize},
xlabel=intensity measure $f_1$, 
ylabel={volume $f_2$},
legend style={cells={anchor=west},font=\scriptsize, at={(axis cs:0.004,0.25)}}]

\addplot[only marks, color=c2,mark=o, thick]	table[x=V1,y=V2]{Data/RM01_wsm.txt};
\addplot[only marks, color=c3,mark=x, thick]	table[x=V1,y=V2]{Data/RM01_moda.txt};

\legend{weighted sum method,
biobjective descent algorithm}

    \node (node15) at (axis cs:0.00628642587582701,	0.119057531107808) {}; 
    \node (node20) at (axis cs:0.00481763826282873,	0.12534762612515) {}; 
    \node (node25) at (axis cs:0.00386989812386961,	0.130932555461196) {}; 
    \node (node30) at (axis cs:0.00323266332817947,	0.13551242350698) {}; 
    \node (node35) at (axis cs:0.00269879329456956,	0.140868670131301) {}; 
    \node (node40) at (axis cs:0.00227968761622789,	0.146152843217973) {}; 
    \node (node45) at (axis cs:0.00194970959124313,	0.151148050655687) {}; 
    \node (node50) at (axis cs:0.00167777126657568,	0.15614121054517) {}; 
    \node (node55) at (axis cs:0.00144919934915186,	0.161205597371257) {}; 
    \node (node60) at (axis cs:0.00123996065313321,	0.166987248743939) {}; 
    \node (node65) at (axis cs:0.00106192720380636,	0.172968952855947) {}; 
    \node (node70) at (axis cs:0.000899411323172636,0.179778866325113) {}; 
    \node (node75) at (axis cs:0.00075101863016599,	0.18766506002071) {}; 
    \node (node80) at (axis cs:0.000615054249408687,0.197090098775131) {}; 
    \node (node85) at (axis cs:0.000486452092855204,0.209286486638988) {}; 
    \node (node90) at (axis cs:0.00036280349850509,	0.226764394316477) {}; 

    \node (MAT05) at (axis cs:0.00132504535554398,  0.169006734401845) {}; 
    \node (MAT06) at (axis cs:0.0011187863031739,	0.174296716601323) {}; 
    \node (MAT07) at (axis cs:0.00109714609009212,	0.177784817157023) {}; 
    \node (MAT08) at (axis cs:0.00106601505390083,	0.180723297031942) {}; 
    \node (MAT09) at (axis cs:0.00104175410037659,	0.182878150355639) {}; 
    \node (MAT10) at (axis cs:0.00104139457627547,	0.184542138196684) {}; 
    \node (MAT11) at (axis cs:0.00102610350522279,	0.185899673259722) {}; 
    \node (MAT12) at (axis cs:0.00101825320934445,	0.187119965813399) {}; 
    \node (MAT13) at (axis cs:0.001008193643443, 	0.188084492988564) {}; 
    \node (MAT14) at (axis cs:0.00102187887867795,	0.188912976115544) {}; 
    \node (MAT15) at (axis cs:0.00101527987417515,	0.189632527479652) {}; 
    \node (MAT16) at (axis cs:0.00100955286419131,	0.190263797410544) {}; 
    \node (MAT17) at (axis cs:0.000998131657727939,	0.190895147595957) {}; 
    \node (MAT18) at (axis cs:0.00101726404184479,	0.191393758150319) {}; 
    \node (MAT19) at (axis cs:0.00101323548155685,	0.191837265856783) {}; 
    \node (MAT20) at (axis cs:0.0010096246238397, 	0.192236567327125) {}; 

    \node (contour35) at ($(node35)+ (2cm,0.8cm)$){\includegraphics[width = 0.15\textwidth]{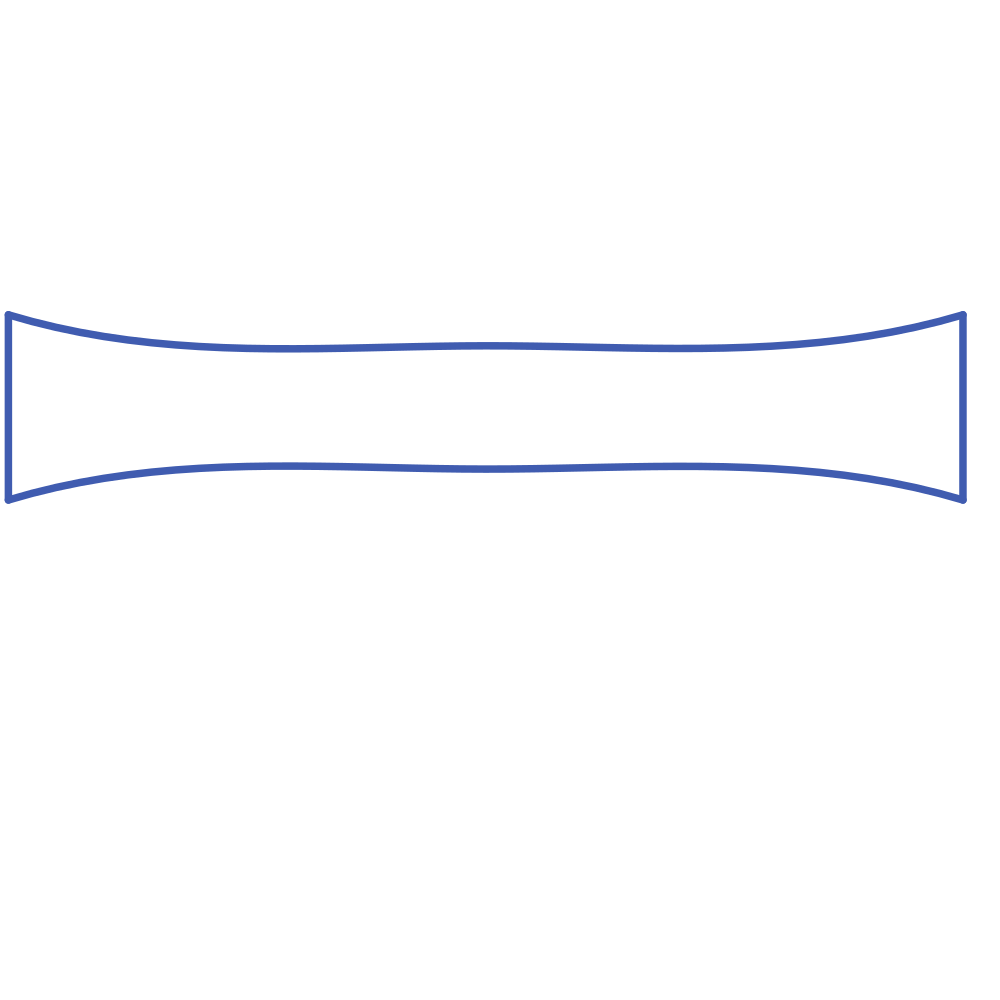}};
    \node (contour90) at  ($(node90)+ (2cm,0.8cm)$)  {\includegraphics[width = 0.15\textwidth]{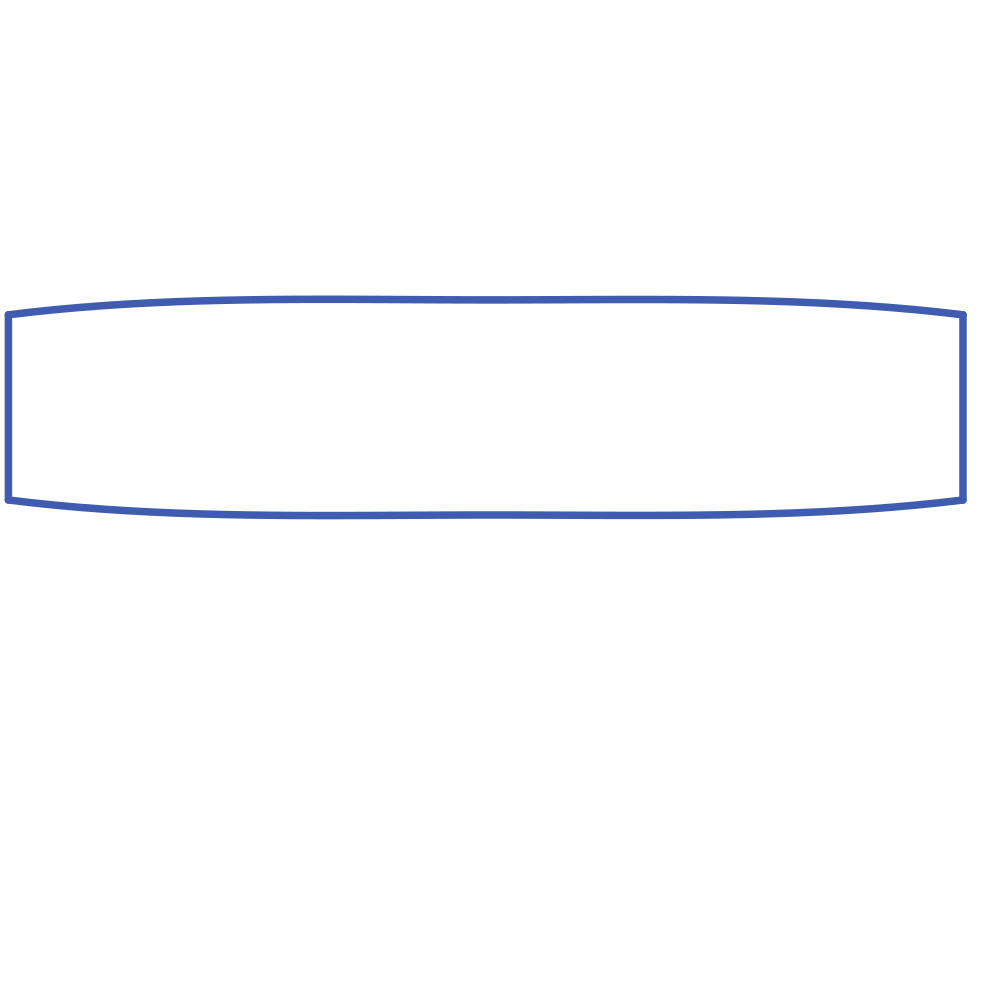}};
    
    \node (contourMAT05) at ($(MAT05)+ (2.1cm,0.5cm)$) {\includegraphics[width = 0.15\textwidth]{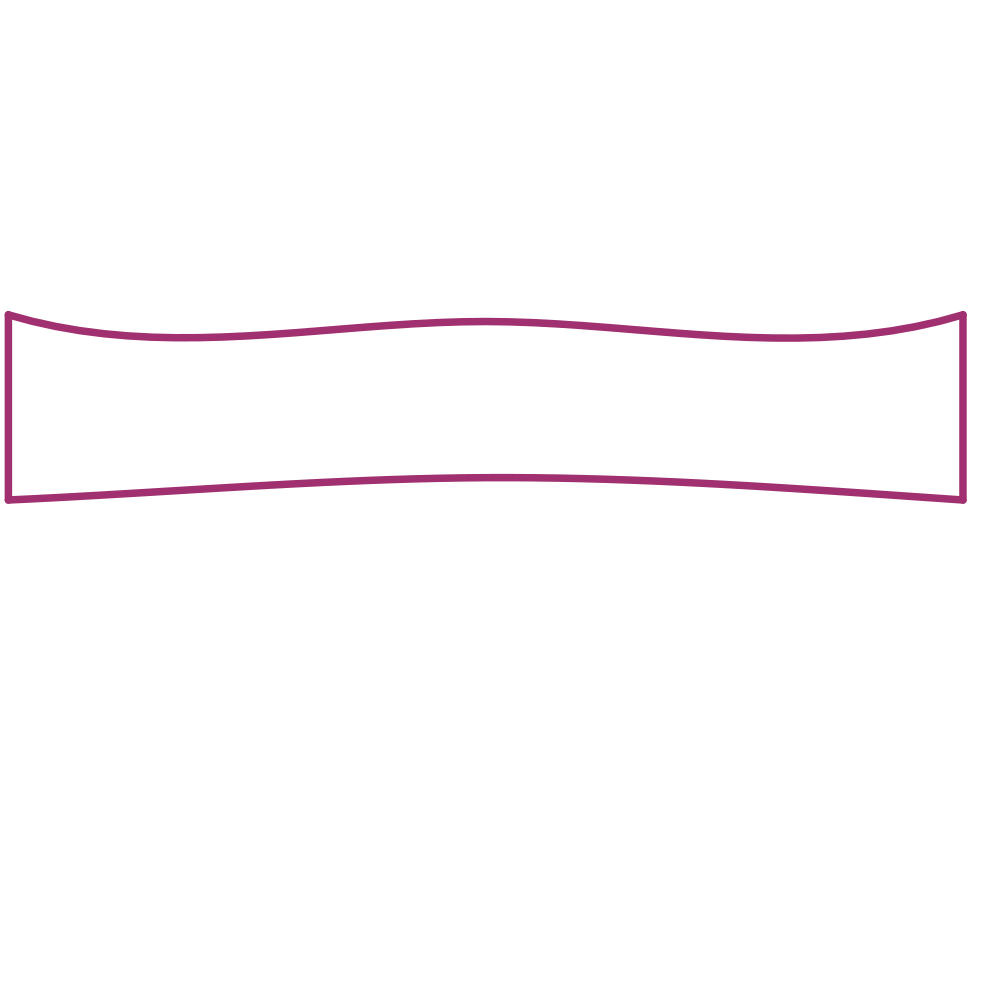}};
    \node (contour60) at ($(node60)+ (2.6cm,-0.1cm)$) {\includegraphics[width = 0.15\textwidth]{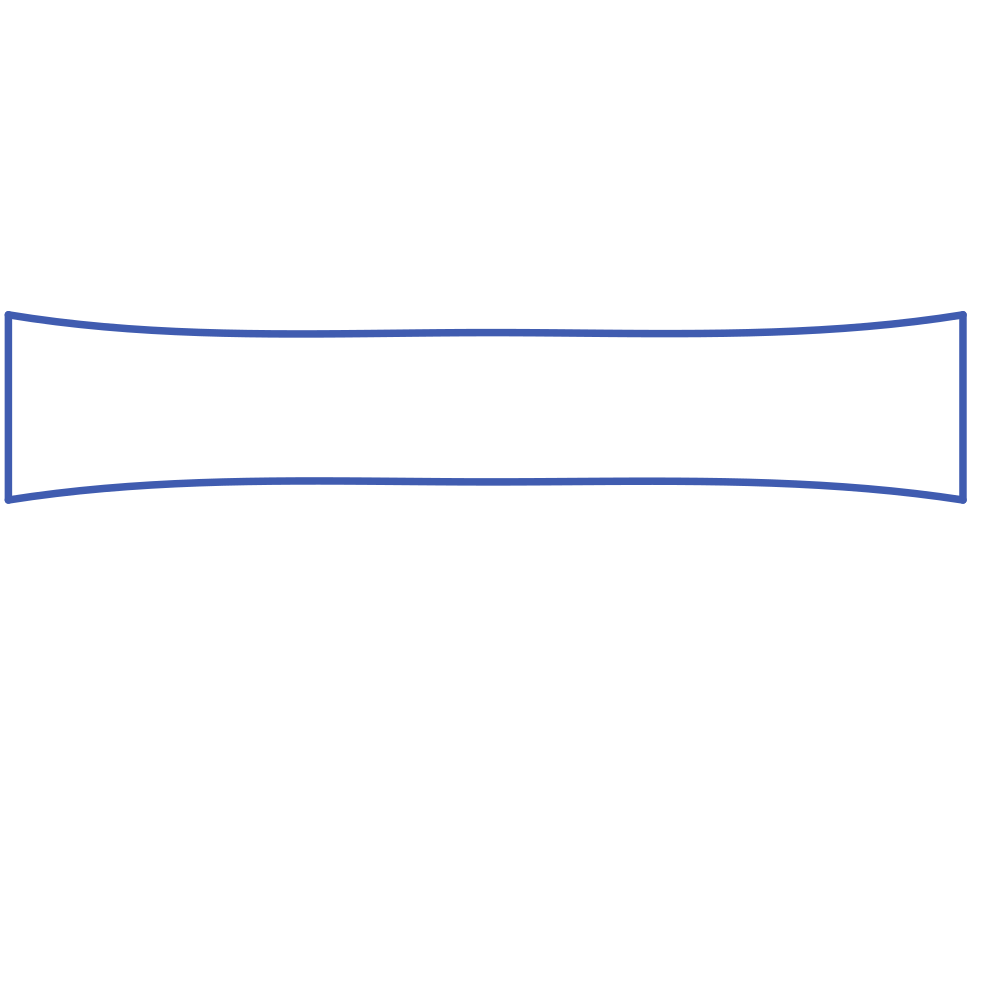}};
    
    \node (contourMAT10) at ($(MAT10)+ (2cm,1.1cm)$){\includegraphics[width = 0.15\textwidth]{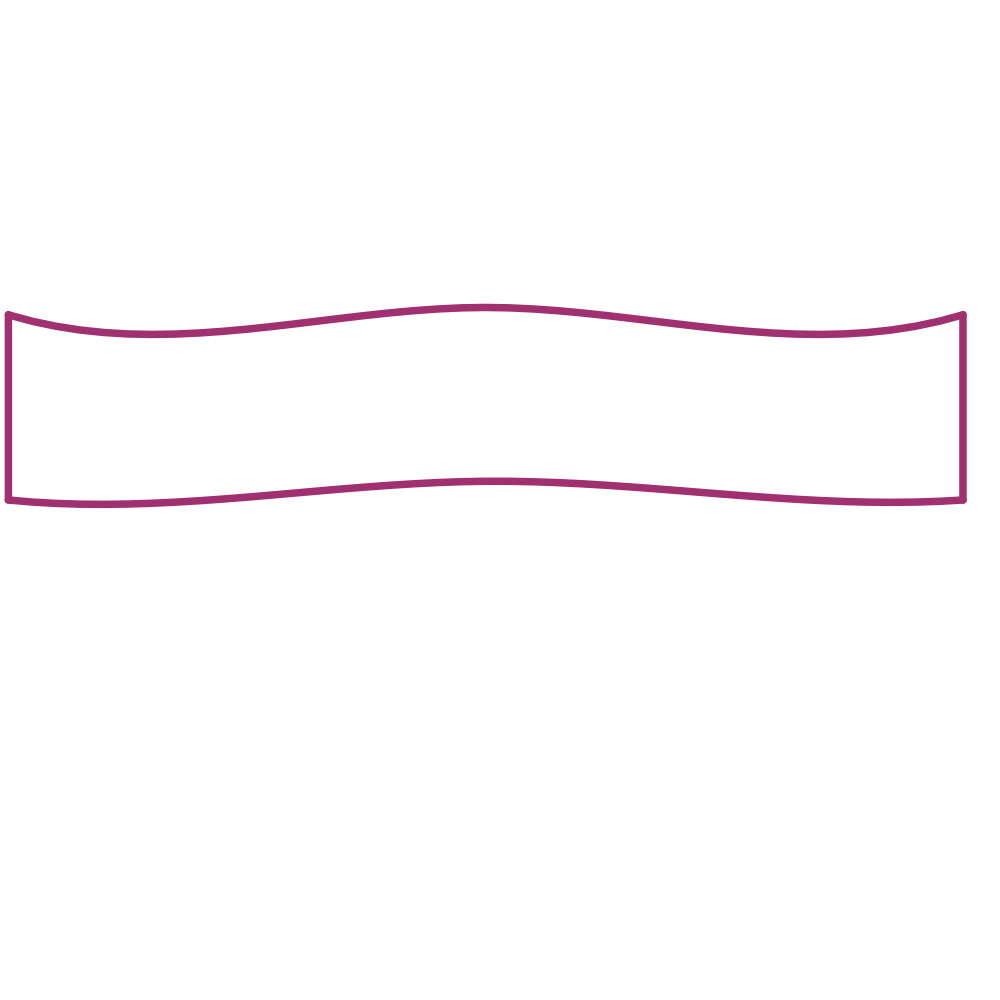}};
    \node (contour70) at ($(node70)+ (2.8cm,0.7cm)$) {\includegraphics[width = 0.15\textwidth]{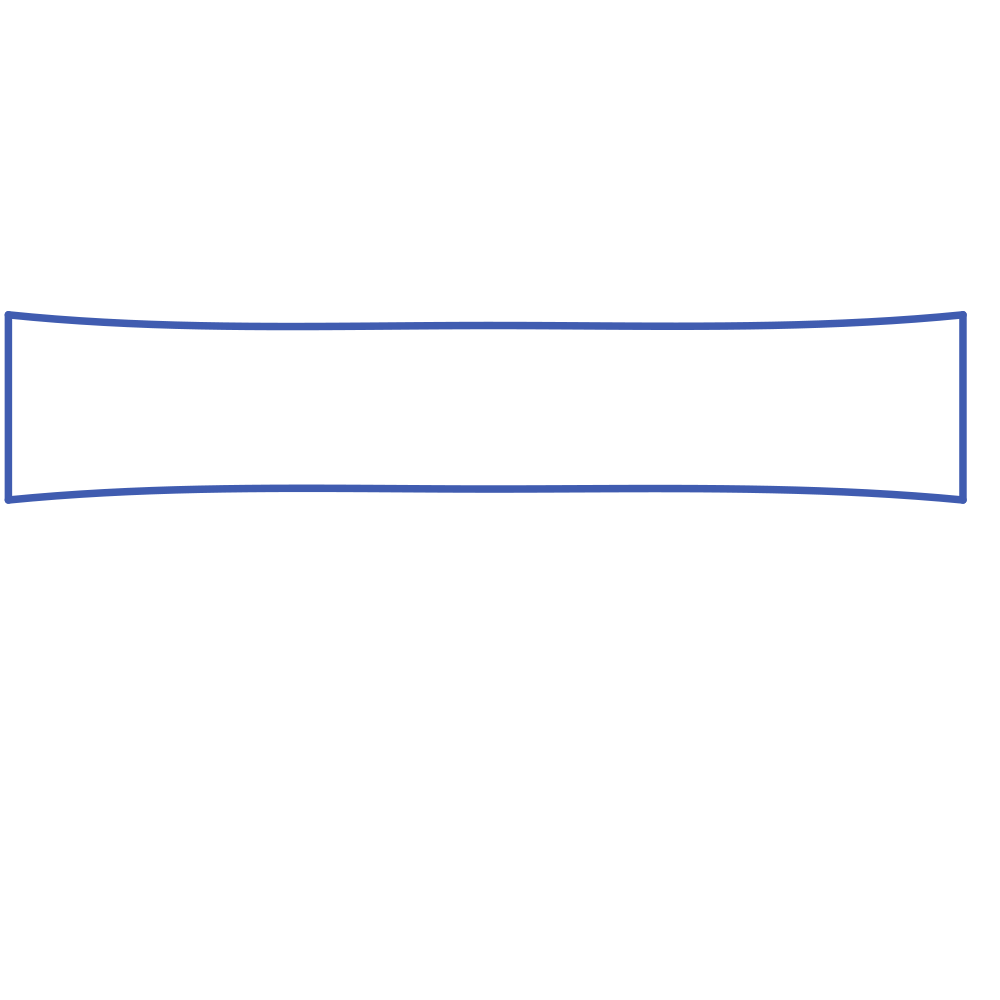}};

    \draw[very thin,color=c2] (node35) -- (contour35.west)   node[midway, above, sloped] {\tiny{$\omega = 0.35$}};
    \draw[very thin,color=c2] (node60) -- (contour60.west)  node[midway, above, sloped] {\tiny{$\omega = 0.6$}};
    \draw[very thin,color=c2] (node70) -- (contour70.west)  node[midway, above, sloped] {\tiny{$\omega = 0.7$}};
    \draw[very thin,color=c2] (node90) -- (contour90.west)  node[midway, above, sloped] {\tiny{$\omega = 0.9$}};
    
    \draw[color=c3,very thin] (MAT05) -- (contourMAT05.west)  node[midway, above, sloped] {\tiny{$\bar{\omega} = 0.5$}};
    \draw[color=c3,very thin] (MAT10) -- (contourMAT10.west)  node[midway, above, sloped] {\tiny{$\bar{\omega} = 1$}};

\end{axis}
\end{tikzpicture}
\caption{\added{Approximated nondominated front for the straight joint. The associated Pareto critical shapes are shown for selected weightings/scalings.}}
\label{objective_space:bended_beam}
\end{figure}

\subsection{An S-Shaped Joint}
\label{subsec:Joint}

A more complex situation is obtained when 
the left and right boundaries are not fixed at the same height, i.e., when an S-shaped joint is to be designed. In our tests, we fix  the right boundary about $0.27\,\text{m}$ lower than the left boundary. 
The starting shape and its $41\times 7$ tetrahedral discretization $X$, that is used for all optimization runs, is shown in Figure~\ref{fig:TC2}. Figure~\ref{fig:TC2_sigma} highlights the stresses that are particularly strong towards the left boundary.  The respective objective values are $f_1(X^{(1)})=1.520058$ (intensity measure) and $f_2(X^{(1)})=0.2$ (volume), respectively. 
As can be expected, the intensity measure (and hence also the PoF) is considerably higher than in the case of the straight joint discussed in Section \ref{subsec:Beam}. 
Despite the significant smoothing induced by the B-spline representation of the initial shape shown in Figure~\ref{fig:RM02_Sigma_Startform}, it has an even higher value of the intensity measure of  $f_1(\gamma^{(1)})=1.910532$ (and hence a higher PoF value), while $f_2(\gamma^{(1)})=0.2$ remains constant.

\begin{figure}[!htb]
	\begin{center}
	\subfloat[Starting shape: Tetrahedral mesh $X$ \label{fig:TC2}]
	    {\includegraphics[width=0.3\textwidth]{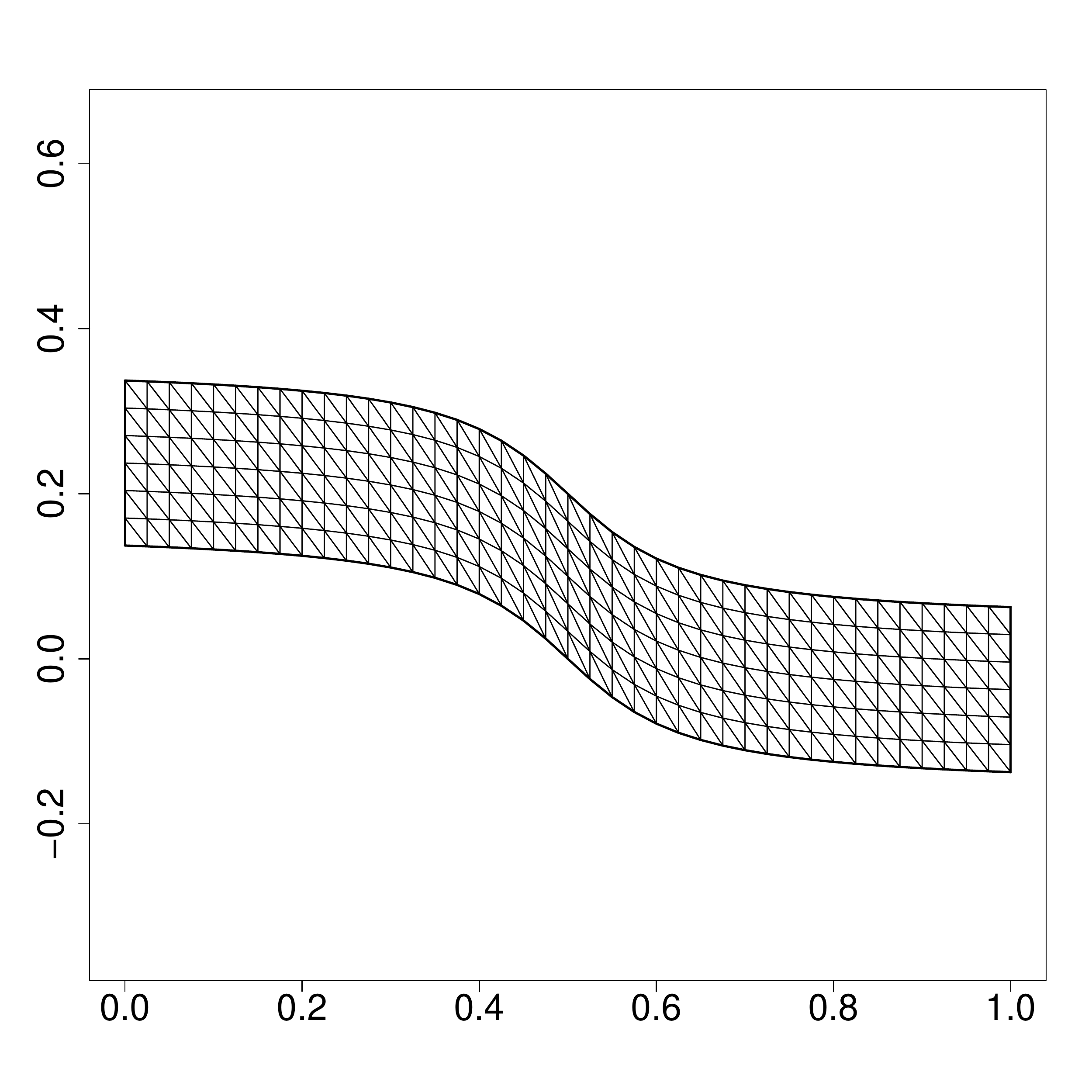}}
	\hspace{\fill}
	\subfloat[Starting shape: Objective values and stresses\label{fig:TC2_sigma}]
	    {\includegraphics[width=0.3\textwidth]{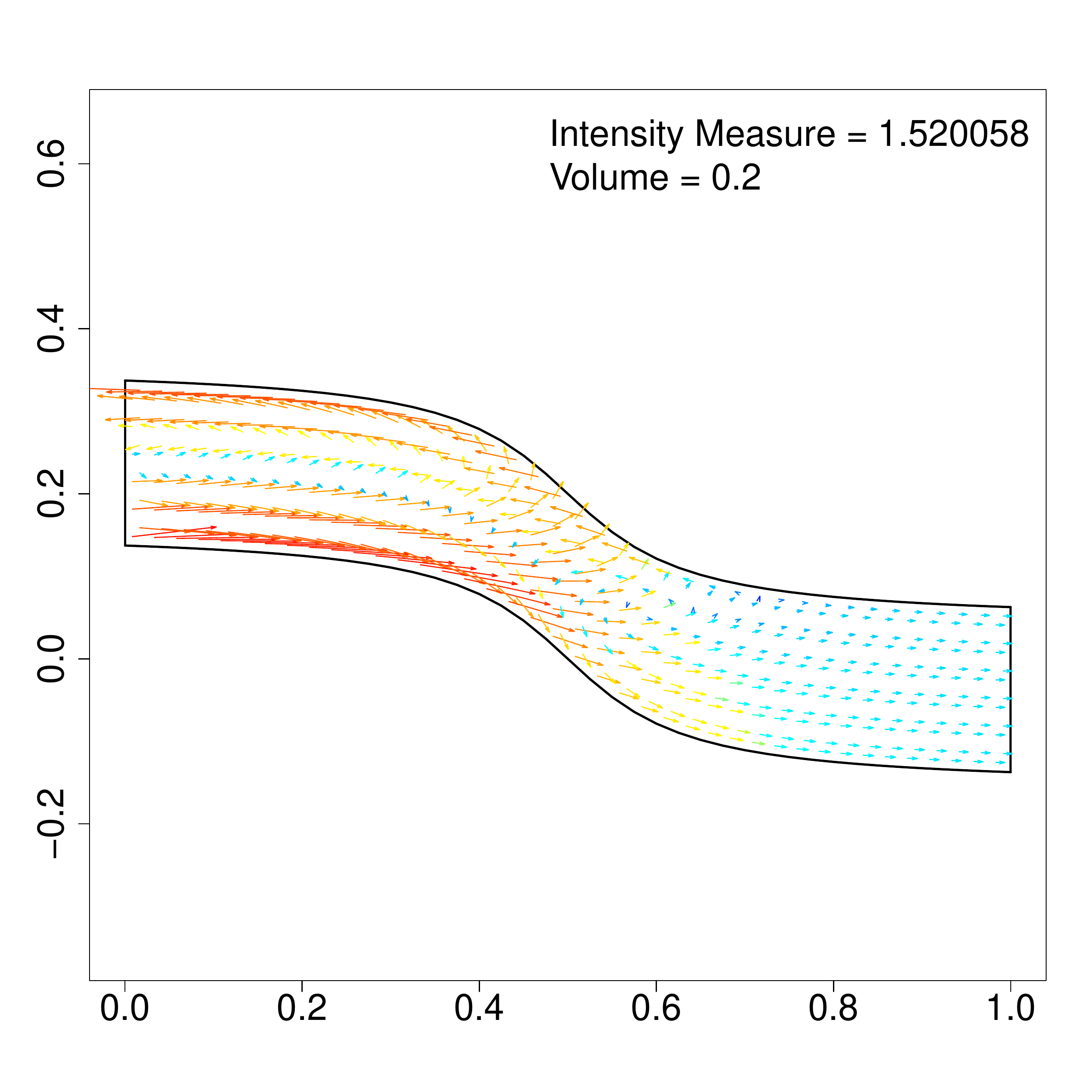}}
	\hspace{\fill}
	\subfloat[Starting shape: Approximation with B-splines\label{fig:RM02_Sigma_Startform}]
	    {\includegraphics[width=0.3\textwidth]{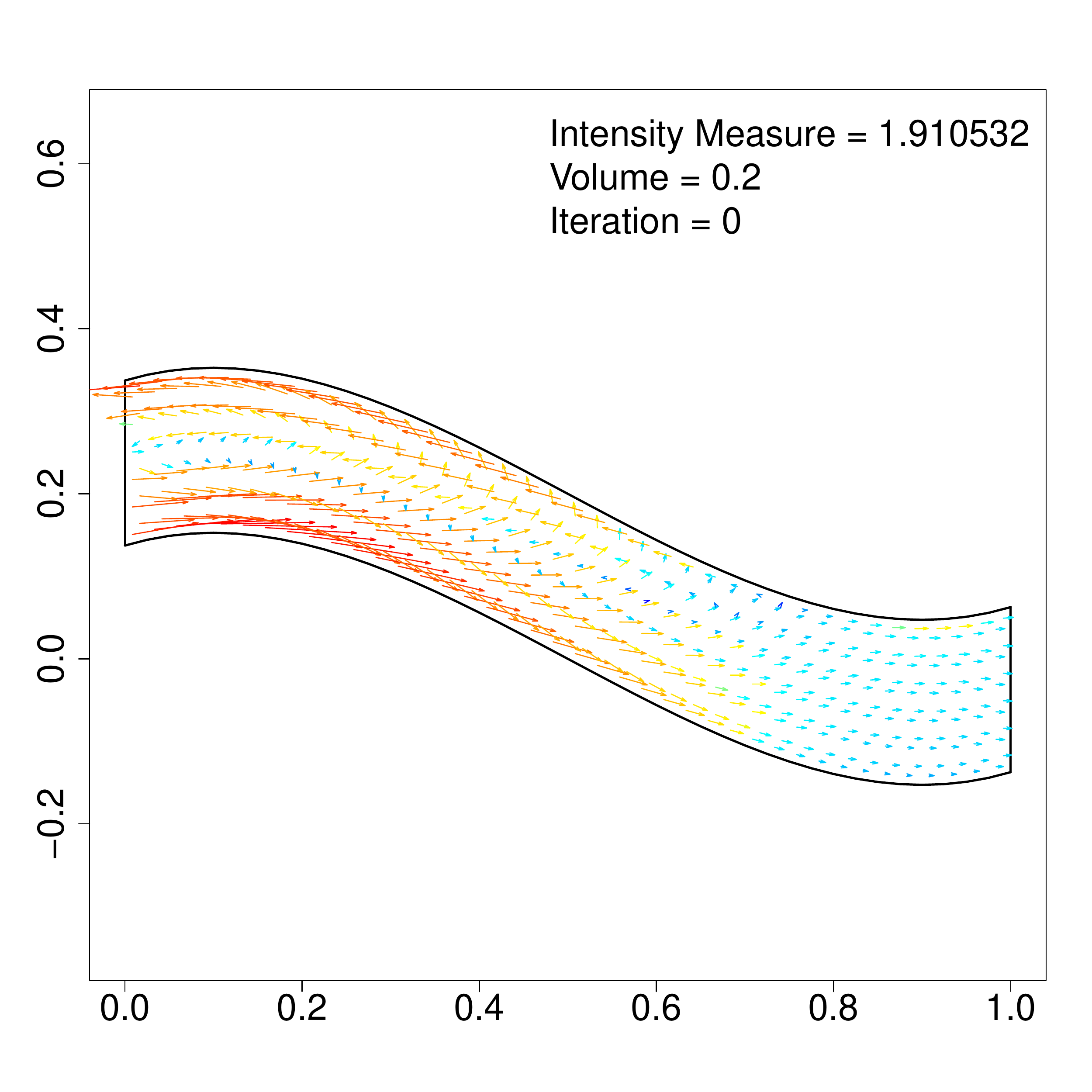}}
	\hspace{\fill}
	\subfloat[\added{MO descent, $\bar{\omega} = 1.1$\label{fig:TC2_MAT11}}]
	    {\includegraphics[width=0.3\textwidth]{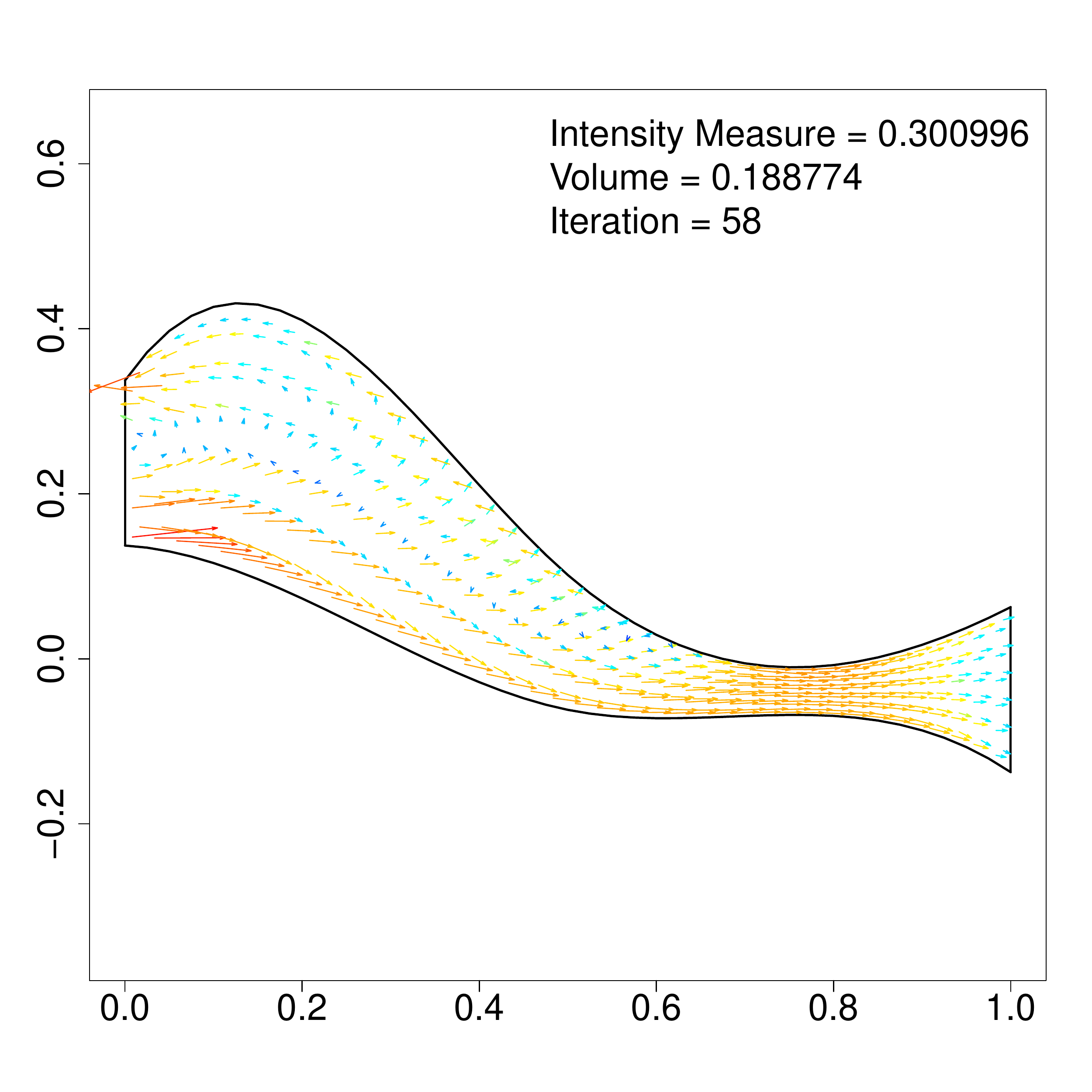}}	
    \hspace{\fill}	
	\subfloat[\added{Weighted sum, $\omega = 0.8$\label{fig:TC2_omega080}}]
	    {\includegraphics[width=0.3\textwidth]{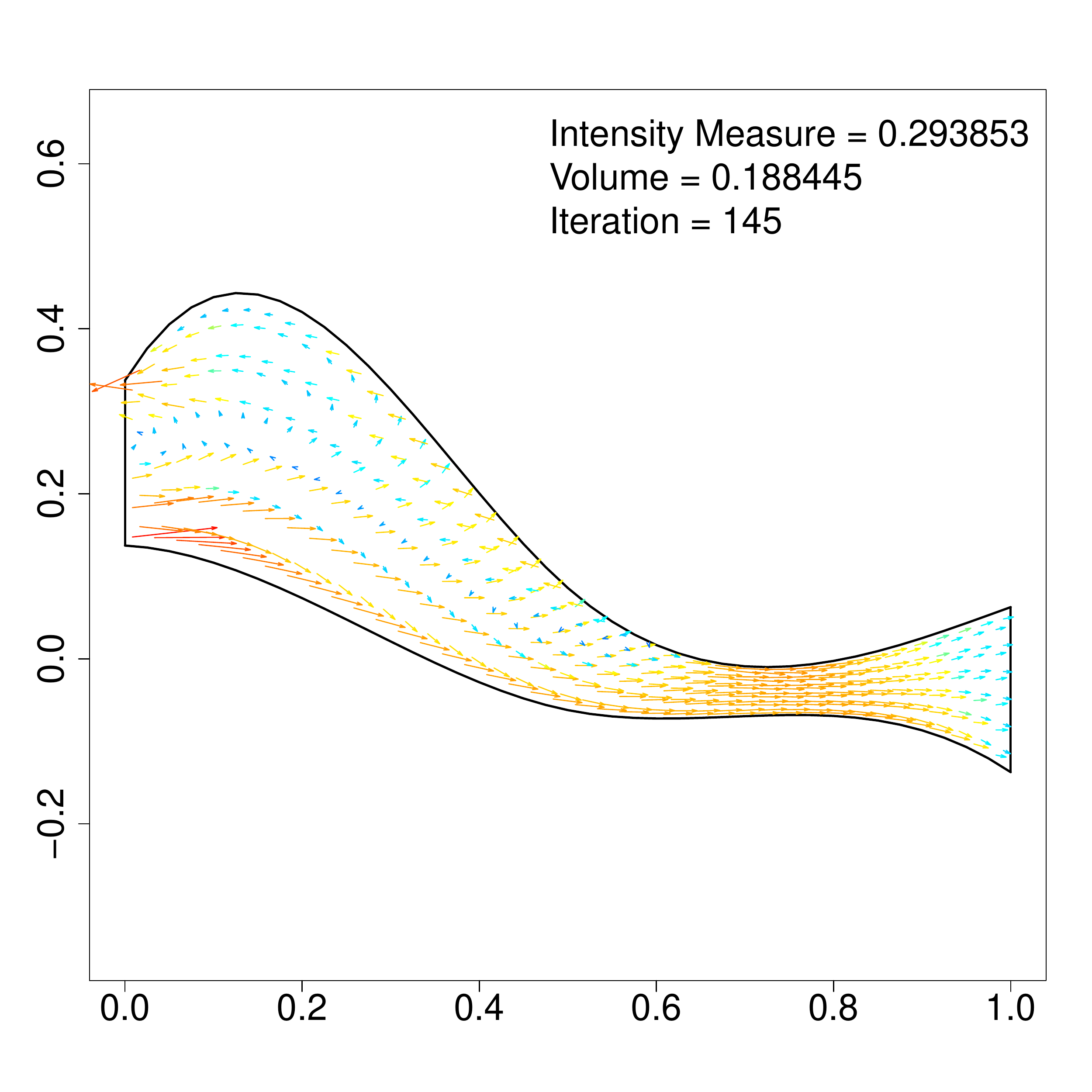}}	
	\hspace{\fill}	
	\subfloat[\added{Weighted sum, $\omega = 0.85$, not converged\label{fig:TC2_omega085}}]
	    {\includegraphics[width=0.3\textwidth]{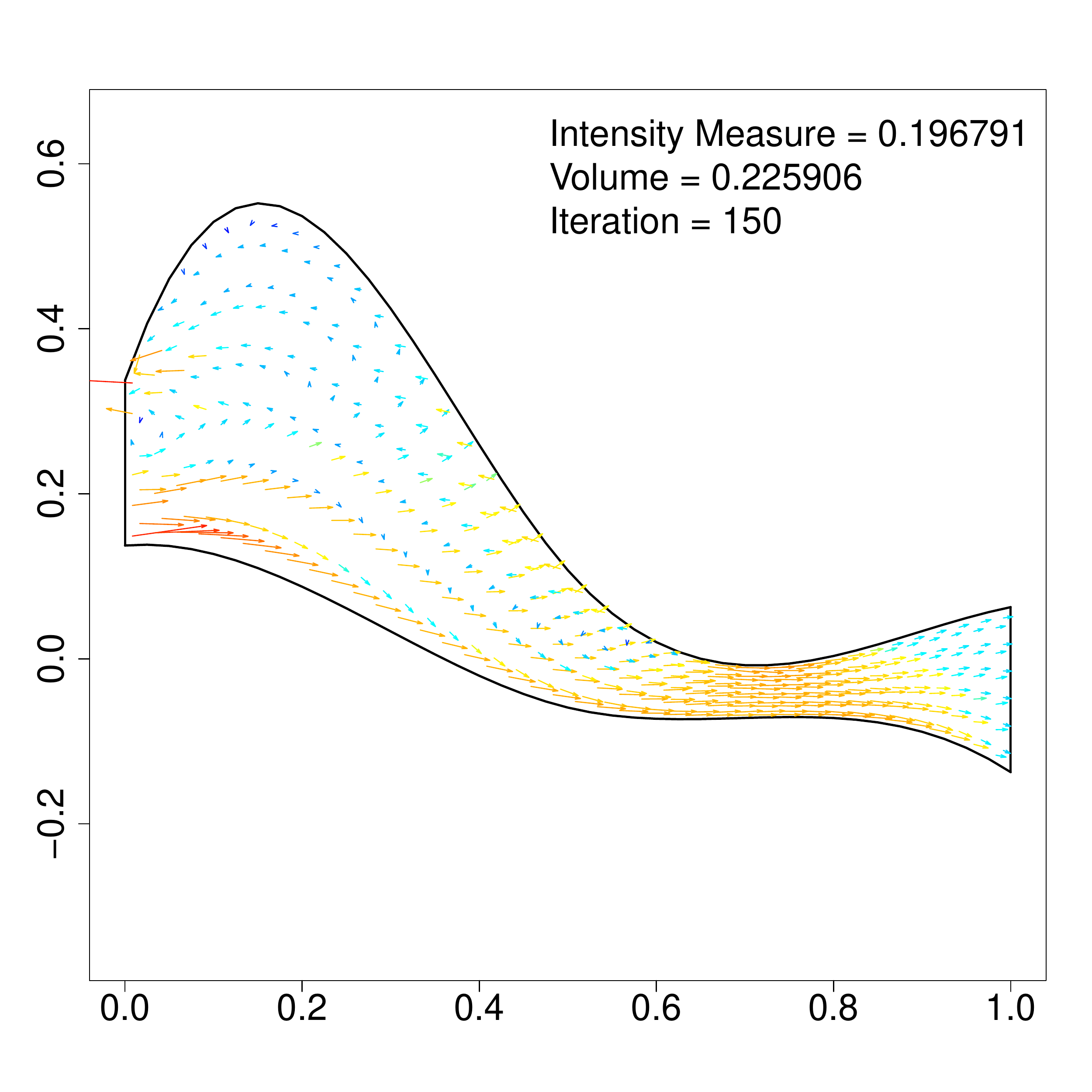}}
	\hspace{0.3\textwidth}
	\hspace{\fill}	
	\caption{S-shaped joint: Starting solution (row 1), \added{two}
	 exemplary Pareto critical solutions (\added{\ref{fig:TC2_MAT11} and \ref{fig:TC2_omega080}})
	\added{, and a not converged solution of the weighted sum method (\ref{fig:TC2_omega085})}.\label{fig:sshaperesults}}
	\end{center}
\end{figure}

\paragraph{Results}
We observe that the resulting shapes resemble the profile of a whale. If we consider 1st principal stress of the stress tensor on the grid points of the initial shape resulting from tensile load, see Figure \ref{fig:RM02_Sigma_Startform}, we observe an anti clockwise eddy in the left part of the joint. The hunch close to the left boundary of the optimized shapes gives room for the occurring stresses and therefore improves the intensity measure and, likewise, the PoF.

Note that, different from the case of the straight rod, we have no prior knowledge on the Pareto optimal shapes. For the solutions shown in Figures~\added{\ref{fig:TC2_MAT11} and \ref{fig:TC2_omega080},} 
we can only guarantee that they are (approximately) Pareto critical, i.e., the respective optimization runs terminated due to the criticality test. 
{Figure~\added{\ref{fig:TC2_omega085}}
shows a 
shape with a \added{significantly}
higher volume of \added{$f_2(X)=0.225906$,}
and with a largely improved intensity measure of \added{$f_1(X)=0.196791$.}
This shape was obtained with the weighted sum method with weight \added{$\omega=0.85$ after $150$ iterations. In this case, the algorithm terminated since it reached the maximum number of iterations and not due to convergence. We observed that all optimization runs of the weighted sum method with $\omega\geq 0.85$ were not converging in this setting. Thus in these cases it is not guaranteed, that the resulting solutions are Pareto critical.}
Note that, given a starting solution with a volume of $0.2$, this shape is not attainable with the biobjective descent algorithm.}

However, there is no guarantee that the computed shapes are Pareto optimal. \added{For example, the shape shown in Figure~\ref{fig:TC2_omega080} obtained with the weighted sum method with weight $\omega=0.8$ achieves objective values of $f_1(X)=0.293853$ and $f_2(X)=0.188445$, and hence slightly dominates the shape shown in Figure~\ref{fig:TC2_MAT11} obtained with the biobjective descent algorithm with scaling parameter $\bar{\omega}=1.1$ that has objective values $f_1(X)=0.300996$ and $f_2(X)=0.188774$.}

Figure~\ref{objective_space:s-joint} summarizes the results of several optimization runs of both Algorithms~\ref{alg:weightedsum} and \ref{alg:fliege} in the objective space.
\added{Note that not all solutions of the weighted sum method lie on the convex hull of the computed points (and are thus not globally optimal for a weighted sum scalarization).}
\added{In some cases, the biobjective descent algorithm also computes dominated points, while in other cases it found solutions that lie even below the convex hull of the weighted sum solutions (see, e.g., the result for $\bar{\omega} = 0.5$ in Figure \ref{objective_space:s-joint}).}

A larger range of alternative objective vectors is, as in the case of the straight rod, obtained with the weighted sum method. 
A cross-test between the two methods, where the final solution of Algorithm~\ref{alg:weightedsum} was used as starting solution for Algorithm~\ref{alg:fliege}, confirms that local Pareto critical solutions were found \added{for $\omega \leq 0.8$.}

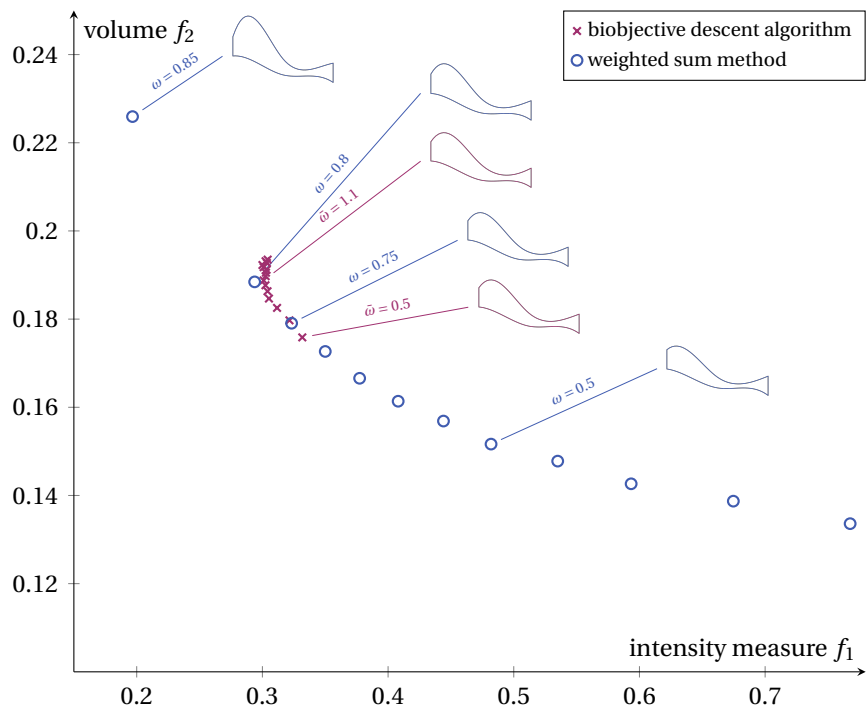
\begin{figure}

	\centering
		\definecolor{c1}{HTML}{3FA137}
		\definecolor{c2}{HTML}{405CB0}
		\definecolor{c3}{HTML}{A0306F}
		\definecolor{c4}{HTML}{C7A748}
		\begin{tikzpicture} 
		\begin{axis}[axis lines=middle,
		width=12cm,
		xmin=0.15,
		xmax=0.78, 
		ymin=0.1,
		ymax=0.25,
		label style={font=\footnotesize}, 
		tick label style={font=\footnotesize},
		xlabel=intensity measure $f_1$, 
		ylabel={volume $f_2$},
		legend style={cells={anchor=west},font=\scriptsize, at={(axis cs:0.78,0.25)}}]

		\addplot[only marks, color=c3,mark=x, thick]	table[x=V1,y=V2]{Data/RM02_moda.txt};
 		\addplot[only marks, color=c2,mark=o, thick]	table[x=V1,y=V2]{Data/RM02_wsm.txt};
		
		\legend{biobjective descent algorithm,
			weighted sum method}
		
\node (node50) at (axis cs:	0.481938338455424,	0.151647772927598) {}; 
		\node (node75) at (axis cs:	0.323246357464224,	0.179059188836481) {}; 
		\node (node80) at (axis cs:	0.293842128614426,	0.188445327938195) {}; 
		\node (node85) at (axis cs:	0.196668758449793,	0.225934698214632) {}; 

		\node (MAT05) at (axis cs:0.331743476110563,  0.175853739723793) {}; 
		\node (MAT06) at (axis cs:0.32154182911223,   0.179716391319067) {}; 
		\node (MAT11) at (axis cs:0.300995884024783,  0.188773551554451) {}; 
		
		\node (contour80) at ($(node80)+(3cm,2.5cm)$) {\includegraphics[width = 0.1\textwidth]{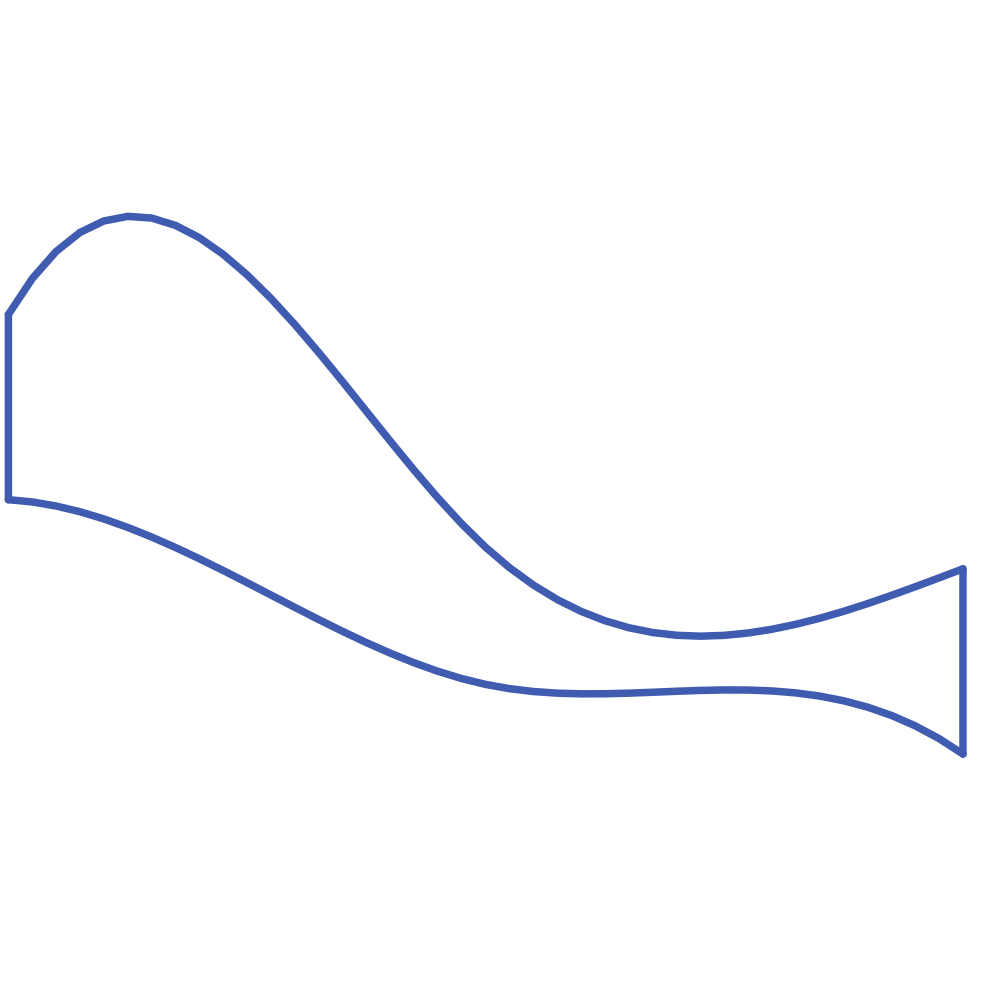}};
		\node (contourMAT11) at  ($(node80)+(3cm,1.6cm)$)  {\includegraphics[width = 0.1\textwidth]{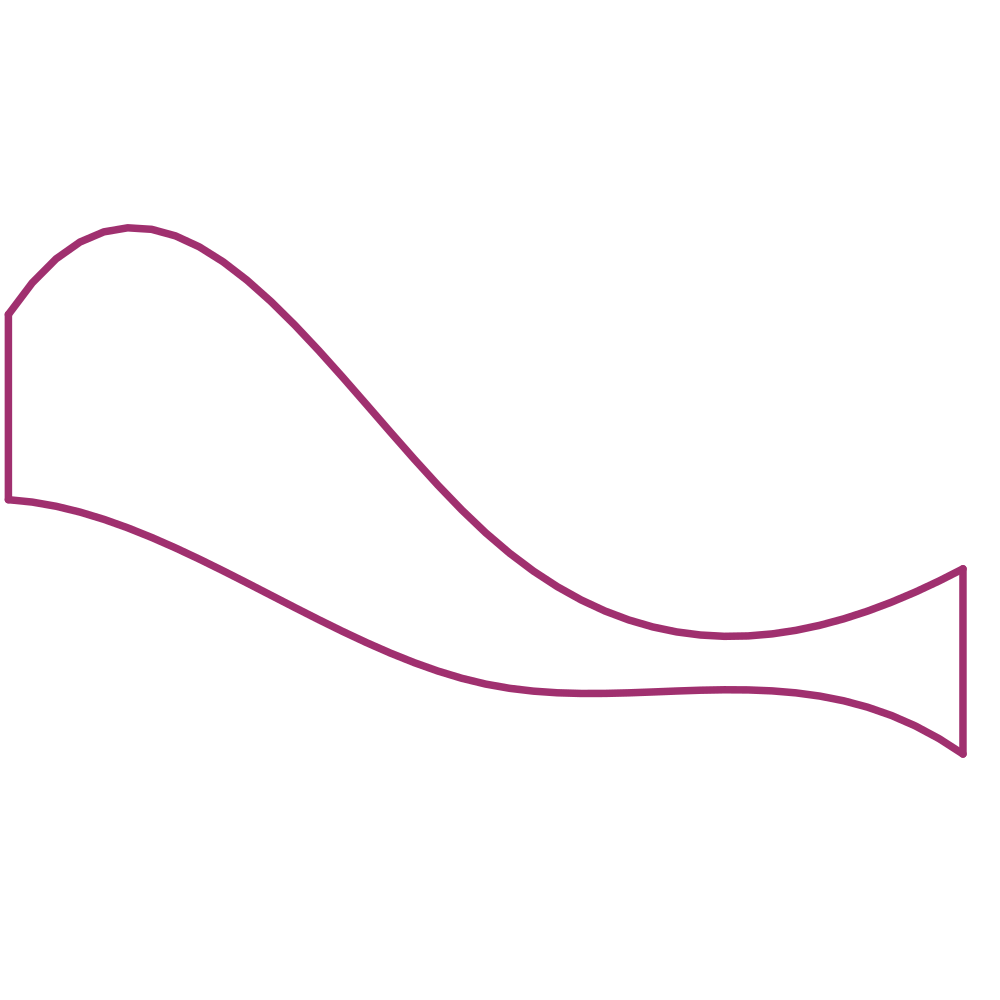}};
		
		\node (contour75) at ($(node75)+(3cm,1.1cm)$) {\includegraphics[width = 0.1\textwidth]{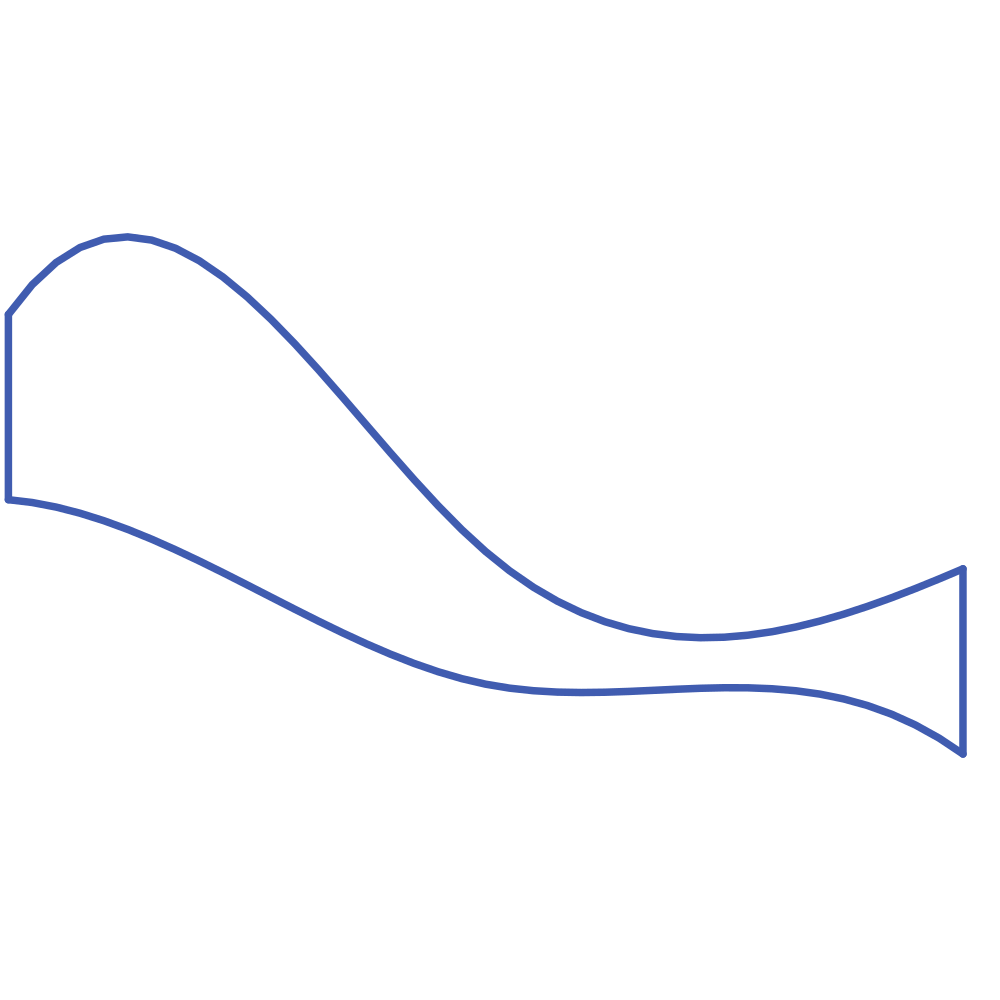}};
		
		\node (contourMAT05) at  ($(MAT05)+(3cm,0.4cm)$) {\includegraphics[width = 0.1\textwidth]{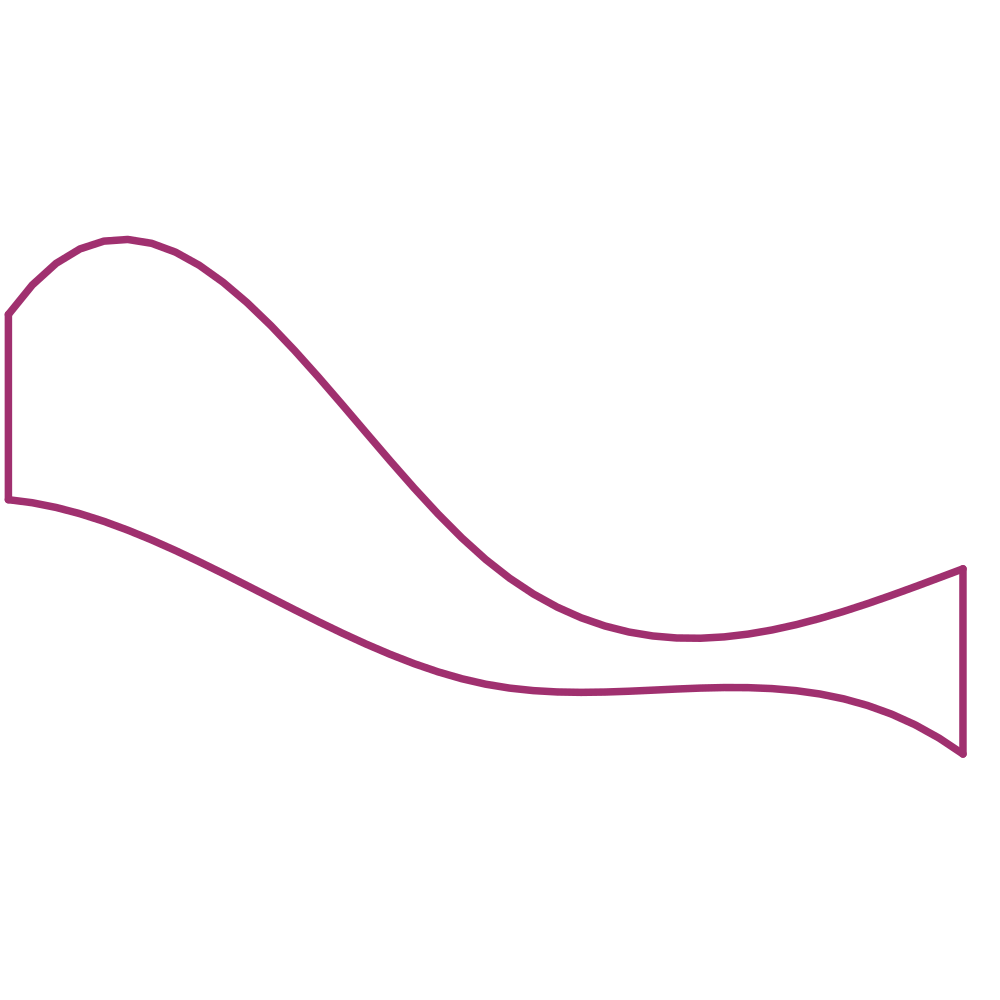}};
		
		\node (contour50) at ($(node50)+(3cm,1cm)$) {\includegraphics[width = 0.1\textwidth]{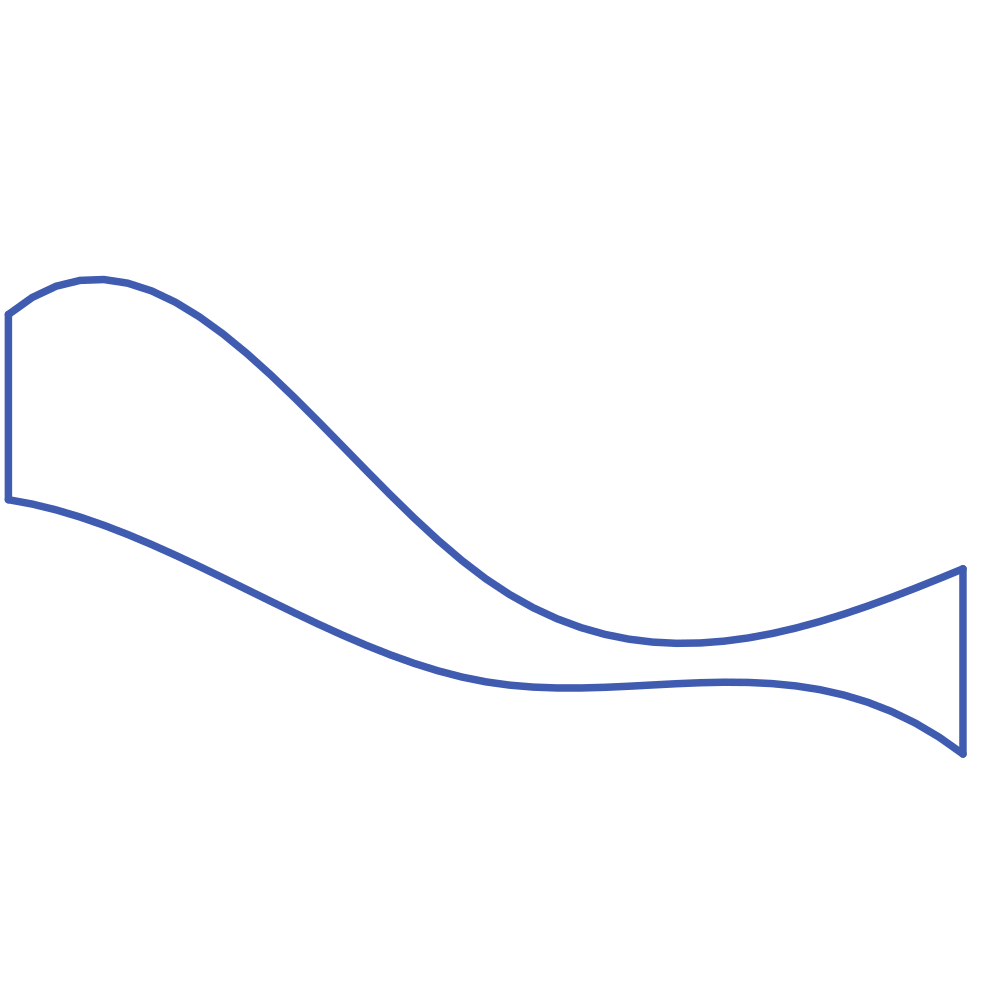}};
		\node (contour85) at ($(node85)+(2cm,0.8cm)$) {\includegraphics[width = 0.1\textwidth]{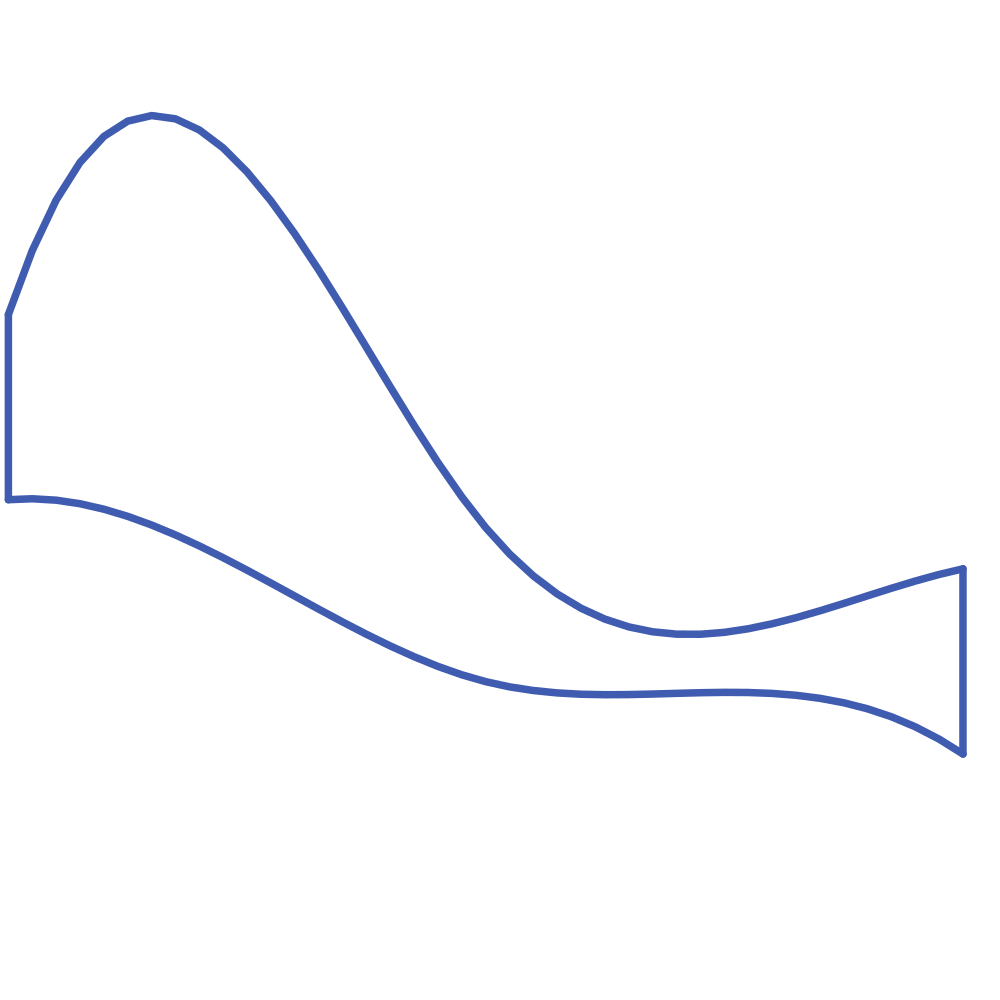}};
		
		\draw[color=c3,very thin] (MAT05) -- (contourMAT05.west)  node[midway, above, sloped] {\tiny{$\bar{\omega} = 0.5$}};
		\draw[color=c3,very thin] (MAT11) -- (contourMAT11.west)  node[midway, above, sloped] {\tiny{$\bar{\omega} = 1.1$}};
		\draw[color=c2,very thin] (node80) -- (contour80.west)  node[midway, above, sloped] {\tiny{$\omega = 0.8$}};
		\draw[color=c2,very thin] (node75) -- (contour75.west)  node[midway, above, sloped] {\tiny{$\omega = 0.75$}};
		\draw[color=c2,very thin] (node50) -- (contour50.west)  node[midway, above, sloped] {\tiny{$\omega = 0.5$}};
	    \draw[color=c2,very thin] (node85) -- (contour85.west)  node[midway, above, sloped] {\tiny{$\omega = 0.85$}};
		
		\end{axis}
\end{tikzpicture}

\caption{\added{Outcome vectors for the S-shaped joint. The associated Pareto critical shapes are shown for selected weightings / scalings.}}
\label{objective_space:s-joint}
\end{figure}


\section{Conclusion and Outlook}
\label{Sec:Outlook}

We have developed a modelling and solution approach  
for biobjective PDE constrained shape optimization of ceramic components. The mechanical integrity of the component on one hand, and the cost of the component on the other hand, were considered as two pivotal optimization criteria. A probabilistic approach was used to assess the mechanical integrity (i.e., the reliability) of the component, which allows, in combination with a finite element discretization and an adjoint approach for gradient computations, the efficient calculation of derivative information. Approximations of the Pareto front were computed using two different approaches: (1) parametric weighted sum scalarizations in combination with a single objective gradient descent method, and (2) a biobjective descent algorithm with parametric scalings of the objective functions. Numerical results for 2D test cases visualize the trade-off between the reliability and the cost, and hence pave the way for an informed selection of a most preferred design. 
A generalization to 3D shapes seems possible and is the next natural step. Moreover, further optimization criteria like, for example, reliability w.r.t.\ other loading scenarios,  minimal natural frequencies,  
and/or efficiency criteria, can be included into a general multiobjective shape optimization problem. 

\section*{Acknowledgement} 
This work was supported by the federal ministry of research and education (BMBF, grant-no:  05M18PXA ) as a part of the GIVEN consortium. 



\end{document}